\documentclass[a4paper,11pt]{article}
\usepackage{amsmath}
\usepackage{amsfonts}
\usepackage{amssymb}
\usepackage{mathrsfs}
\usepackage{amscd}
\usepackage{pb-diagram}
\usepackage{amsthm}
\usepackage{color}
\usepackage[all]{xy}
\usepackage{graphicx}
\usepackage{url}
\usepackage{enumerate}
\usepackage{tikz-cd}
\numberwithin{equation}{section} 
\usepackage{caption}
\DeclareUnicodeCharacter{2212}{-}
\usepackage{tikz}

\usepackage{titlesec}
\titleformat{\subsection}[runin]{\normalsize\bfseries}{\thesubsection}{5pt}{}

\usepackage{tikz}          
\usetikzlibrary{matrix, arrows, decorations.pathmorphing}

\newcommand{\moment}{
\textup{\textbf{J}}
}

\author{Mathieu Molitor
\\
\small{\it{e-mail:}}\,\,\url{pergame.mathieu@gmail.com}}
\title{Weyl group of the group of holomorphic isometries of a K\"{a}hler toric manifold\\
}
\date{}

\begin{document}

\theoremstyle{definition}
\newtheorem{lemma}{Lemma}[section]
\newtheorem{definition}[lemma]{Definition}
\newtheorem{proposition}[lemma]{Proposition}
\newtheorem{corollary}[lemma]{Corollary}
\newtheorem{theorem}[lemma]{Theorem}
\newtheorem{remark}[lemma]{Remark}
\newtheorem{example}[lemma]{Example}
\newtheorem{setting}[lemma]{Setting }

\bibliographystyle{alpha}

\maketitle 


\begin{abstract}

	We compute the Weyl group (in the sense of Segal) of the group of holomorphic 
	isometries of a K\"{a}hler toric manifold with real analytic K\"{a}hler metric. 
\end{abstract}




\section{Introduction}

	Let $N$ be a K\"{a}hler toric manifold, that is, a compact, connected K\"{a}hler manifold $N$ 
	equipped with an effective Hamiltonian and isometric torus action $\Phi:\mathbb{T}^{n}\times N\to N$, 
	where $n$ is the complex dimension of $N$ and $\mathbb{T}^{n}=\mathbb{R}^{n}/\mathbb{Z}^{n}$ is the $n$-dimensional real torus. 
	We assume that the K\"{a}hler metric $g$ is real analytic. 
	Let $\textup{Aut}(N,g)$ be the group of holomorphic isometries of $N$ endowed with the compact-open topology;
	it is a compact Lie group. 

	The present paper is concerned with the explicit computation of the Weyl group of $\textup{Aut}(N,g)$. 
	Because $\textup{Aut}(N,g)$ can fail to be connected, we adopt of definition of the Weyl group 
	given by Segal \cite{segal} for compact, but not necessarily connected Lie groups, that we now recall. 
	Let $G$ be a compact Lie group, not necessarily connected. A closed subgroup $S\subseteq G$ is called a \textit{Cartan subgroup}
	if it contains an element whose powers are dense in $S$, and if $S$ is of finite index in its normalizer 
	$N(S)=\{g\in G\,\,\vert\,\,gSg^{-1}=S\}$. The finite group $W(S)=N(S)/S$ is, by definition, the \textit{Weyl group} of $S$ 
	(in the sense of Segal). If $G$ is connected, then Cartan subgroups are precisely the 
	maximal tori of $G$ (see, e.g., \cite{tammo}, Chapter IV). Therefore 
	Cartan subgroups are natural generalizations of maximal tori for 
	nonconnected compact Lie groups. 

	Returning to our main concern, the group $\textup{Aut}(N,g)$,
%
	we show that the group $S\subset \textup{Aut}(N,g)$
	of transformations of $N$ 
	associated to the torus action is a Cartan subgroup of 
	$\textup{Aut}(N,g)$. Our main theorem is a description of the corresponding 
	Weyl group $W=W(S)$ (Theorem \ref{nefknkenknkwndk}). Our results are expressed in the language of 
	\textit{dually flat manifolds}, with which not all readers may be familiar.
	A dually flat manifold is a triple $(M,h,\nabla)$, where $M$ is a manifold, 
	$h$ is a Riemannian metric on $M$ and $\nabla$ is a flat connection 
	on $M$ (not necessarily the Riemannian connection) whose dual connection 
	is also flat (see Section \ref{nkwwnkwnknknfkn}). Let $\textup{Diff}(M,h,\nabla)$ be the group of 
	isometries of $(M,h)$ that are also affine with respect to $\nabla$. Our main result says that $W$ 
	can be realized as a group of the form $\textup{Diff}(M,h,\nabla)$, where $(M,h,\nabla)$ is an appropriate dually flat manifold. 
	The latter may appear in different guises and forms, leading to different (but isomorphic) 
	realizations of $W$. Let us look at the following two examples, that are particular instances of our main theorem.

	\noindent\textbf{Symplectic realization.} Let $\moment:N\to \mathbb{R}^{n}\cong \textup{Lie}(\mathbb{T}^{n})^{*}$ be a momentum map 
	for the torus action, and let $\Delta=\moment(N)\subset \mathbb{R}^{n}$ be the corresponding 
	momentum polytope. Let $\Delta^{\circ}$ denote the topological interior of $\Delta$.
	On the open set $N^{\circ}$ of points $p\in N$ where the torus action is free, 
	$\moment$ is a $\mathbb{T}^{n}$-invariant submersion and hence 
	there is a unique Riemannian metric $k$ on $\Delta^{\circ}=\moment(N^{\circ})$ 
	that turns $\moment:(N^{\circ},g)\to (\Delta^{\circ},k)$ into a Riemannian 
	submersion. Let $\nabla^{\textup{flat}}$ be the canonical flat connection on $\mathbb{R}^{n}$ and let $\nabla^{k}$ be the 
	dual connection of $\nabla^{\textup{flat}}$ on $\Delta^{\circ}$ with respect to $k$. 
	The triple $(\Delta^{\circ},k,\nabla^{k})$ is a dually flat manifold (this follows from the well-known 
	fact that $k$ is the Hessian of some potential 
	$\varphi:\Delta^{\circ}\to \mathbb{R}$). Then $W$ is isomorphic 
	to $\textup{Diff}(\Delta^{\circ},k,\nabla^{k})$. 
	One may regard this isomorphism as a \textit{symplectic realization} of $W$. 

	\noindent\textbf{Holomorphic realization.} Another way to describe $W$ is by using the holomorphic extension 
	$\Phi^{\mathbb{C}}:(\mathbb{C}^{*})^{n}\times N\to N$ of the torus action, where $(\mathbb{C}^{*})^{n}$ 
	is the algebraic torus (see, e.g., \cite{ishida,molitor-toric}). Given $p\in N^{\circ}$, 
	the map $(\mathbb{C}^{*})^{n}\to N^{\circ}$, $z\mapsto \Phi^{\mathbb{C}}(z,p)$ is a biholomorphism and hence 
	one may identify $N^{\circ}$ with $(\mathbb{C}^{*})^{n}$. 
	Let $\sigma:(\mathbb{C}^{*})^{n}\to\mathbb{R}^{n}$ be defined by $\sigma(z)=(\tfrac{\ln(|z_{1}|)}{2\pi},..., 
	\tfrac{\ln(|z_{n}|)}{2\pi})$, where $z=(z_{1},...,z_{n})\in (\mathbb{C}^{*})^{n}$. It can be shown 
	that there is a Riemannian metric $h$ on $\mathbb{R}^{n}$ that turns $\sigma$ into a Riemannian submersion. 
	Then $W$ is isomorphic to $\textup{Diff}(\mathbb{R}^{n},h,\nabla^{\textup{flat}})$. 
	One may regard this isomorphism as a \textit{holomorphic realization} of $W$. 	

	Both realizations can be treated in a unified way by adopting the language of \textit{torifications of 
	dually flat manifolds}, which was recently introduced in \cite{molitor-toric}. The word ``torification" refers to a geometric 
	contruction that associates to a given dually flat manifold $(M,h,\nabla)$ satisfying certain properties, 
	a K\"{a}hler manifold $N$
	(not necessarily compact) equipped with an effective holomorphic and isometric torus action $\Phi:\mathbb{T}^{n}\times N\to N$, 
	where $n$ is the real (resp. complex) dimension of $M$ (resp. $N$). The manifold $N$ is then called 
	the \textit{torification} of $M$. For instance, every K\"{a}hler toric manifold $N$ is the torification of its momentum polytope 
	$(\Delta^{\circ},k,\nabla^{k})$ (\textit{symplectic point of view}), but also the torification of the dually flat manifold 
	$(\mathbb{R}^{n},h,\nabla^{\textup{flat}})$ associated to the holomorphic extension of the torus action 
	(\textit{holomorphic point of view}). Both points of view are equivalent, because there exists an 
	isomorphism of dually flat manifolds between $(\Delta^{\circ},k,\nabla^{k})$ and 
	$(\mathbb{R}^{n},h,\nabla^{\textup{flat}})$. For more details, see \cite{molitor-toric} and 
	Section \ref{nknkneknknfwknk} below.

	In its most general form, our main result states that for a compact and connected torification $N$
	of a dually flat manifold $(M,h,\nabla)$ with real analytic K\"{a}hler 
	metric $g$, the Weyl group $W(S)$ of $\textup{Aut}(N,g)$ is isomorphic to $\textup{Diff}(M,h,\nabla)$ 
	(Theorem \ref{nefknkenknkwndk}). 
	The main technical tool to prove this result is a lifting procedure introduced in \cite{molitor-toric}.  
	If $(N,g)$ and $(N',g')$ are compact torifications of $(M,h,\nabla)$ and $(M',h',\nabla')$, 
	respectively, with real analytic K\"{a}hler metrics $g$ and $g'$, 
	then every affine isometric immersion $f:M\to M'$ can be lifted to a K\"{a}hler immersion $\widetilde{f}:N\to N'$. Moreover, 
	this lift is equivariant relative to some reparametrization of the torus $\mathbb{T}^{n}$ (see Section \ref{neknkenkfnkn}). When 
	$M=M'$ and $f$ is an affine isometry, this implies that $\widetilde{f}$ is in the normalizer $N(S)$ of 
	$S=\{\Phi_{a}:N\to N\,\vert\,a\in \mathbb{T}^{n}\}$, 
	from which it follows that $N(S)$ is isomorphic to the semidirect product 
	$\mathbb{T}^{n}\rtimes \textup{Diff}(M,h,\nabla)$ (see Theorem \ref{nfkwndkefnknkwn} and 
	Lemma \ref{nefkkwknefknwknknfenkdfndknk}). 

	A particularly interesting case occurs when $M=\mathcal{E}$ is a statistical manifold 
	of exponential type defined over a finite set $\Omega=\{x_{1},...,x_{m}\}$ (see Section \ref{nknkefnknk}). 
	Because $\mathcal{E}$ is a statistical manifold, it is naturally endowed 
	with a Riemannian metric $h_{F}$, the \textit{Fisher metric}, and a connection $\nabla^{(e)},$ called \textit{exponential 
	conection}, and it is well-known that $(\mathcal{E},h_{F},\nabla^{(e)})$ is a dually flat manifold \cite{Amari-Nagaoka}.
	For example, the family $\mathcal{B}(n)$ of Binomial distributions $p(k)=\binom{n}{k}q^{k}(1-q)^{n-k}$ 
	defined over $\Omega=\{0,1,...,n\}$ is a 1-dimensional exponential family parametrized by $q\in (0,1)$. 

	Under mild assumptions, we show that $\textup{Diff}(\mathcal{E},h_{F},\nabla^{(e)})$ coincides with 
	the group of permutations of $\mathcal{E}$ (Proposition \ref{nfekwnknfeknk}). Here, a permutation of $\mathcal{E}$ means 
	a diffeomorphim $\varphi:\mathcal{E}\to \mathcal{E}$ 
	of the form $\varphi(p)(x_{i})=p(x_{\sigma(i)})$, where $\sigma$ is a permutation of $\{1,....,m\}$. 
	A direct consequence is that if $N$ is a compact torification of $\mathcal{E}$ with real analytic K\"{a}hler metric $g$, 
	then the Weyl group $W(S)$ of $\textup{Aut}(N,g)$ is isomorphic to the group of permutations of $\mathcal{E}$. 
	For example, let $\Omega=\{x_{1},...,x_{m}\}$ be a finite set, and let $\mathcal{P}_{m}^{\times}$ be the set of maps 
	$p:\Omega\to \mathbb{R}$ satisfying $p(x)>0$ for all $x\in \Omega$ and 
	$\sum_{x\in \Omega}p(x)=1$. Then $\mathcal{P}_{m}^{\times}$ 
	is an exponential family of dimension $m-1$ whose group of permutations is isomorphic to the group $\mathbb{S}_{m}$ 
	of permutations of $m$ objects. In \cite{molitor-toric}, it is proven that the complex projective space 
	$\mathbb{P}_{m-1}(c)$ of complex dimension $m-1$ 
	and holomorphic sectional curvature $c=1$ is a torification 
	of $\mathcal{P}_{m}^{\times}$. The group of holomorphic isometries of $\mathbb{P}_{m-1}(c)$ is the projective 
	unitary group $\textup{PU}(m)$, which is connected. Thus we obtain a very simple proof that 
	the Weyl group of $\textup{PU}(m)$ is isomorphic to $\mathbb{S}_{m}$.

	We conclude this introduction with a few remarks on the physical significance of our results. 
	The present paper is motivated by the geometrization program of Quantum 
	Mechanics that we pursued in previous works \cite{Molitor2012,Molitor-exponential,Molitor2014,Molitor-2015,molitor-toric}, 
	and more recently, in \cite{molitor-spectral}. Very roughly, we put forth the idea that (finite dimensional) Quantum Mechanics 
	can be reformulated in a geometric and informational fashion, in which the usual Hilbert space 
	$\mathcal{H}=\mathbb{C}^{n}$ is replaced by the torification $N$ of an appropriate exponential family $\mathcal{E}$. The choice 
	of $\mathcal{E}$ depends on the quantum experiment at hand. For instance, the case 
	$\mathcal{E}=\mathcal{B}(n)$ (Binomial distributions) 
	and $N=\mathbb{P}_{1}(\tfrac{1}{n})$ yields the mathematical description of the spin of a non-relativistic quantum particle 
	\cite{Molitor-exponential,molitor-spectral}. For the generic case, one replaces $\mathcal{H}=\mathbb{C}^{n+1}$ by the 
	complex projective space $N=\mathbb{P}_{n}(1)$, which is the torification of $\mathcal{E}=\mathcal{P}_{n+1}^{\times}$; 
	this leads to Geometric Quantum Mechanics \cite{Ashtekar,Spera-geometric}. In this circle of ideas, the Weyl group $W$ of 
	$\textup{Aut}(N,g)$ appears naturally due to the need for a geometric analogue of the spectral decomposition theorem 
	for Hermitian matrices.  Such analogue is presented in \cite{molitor-spectral}. Very roughly, 
	it says that a smooth function $f:N\to \mathbb{R}$ is K\"{a}hler\footnote{A K\"{a}hler function on a K\"{a}hler manifold $N$
	is a smooth real-valued function $f$ whose Hamiltonian vector field $X_{f}$ is a Killing vector field. In the context of Geometric 
	Quantum Mechanics \cite{Ashtekar,Spera-geometric}, K\"{a}hler functions on $N=\mathbb{P}_{n}(c)$ play the role of observables.} 
	if and only if $f$ can be written as an expectation of the form 
	$f(p)=\mathbb{E}_{K(\phi(p))}(X)$ for all $p\in N$, 
	where $K:N\to \overline{\mathcal{E}}$ is a map from $N$ to an appropriate superset of 
	$\mathcal{E}$, $\phi\in \textup{Aut}(N,g)$ and $X$ is an appropriate random variable. 
	One of the main results of \cite{molitor-spectral} is that the image of $X$ is independent of the decompostion 
	$f(p)=\mathbb{E}_{K(\phi(p))}(X)$, allowing us to define the spectrum of a K\"{a}hler function $f$ as the set 
	$\textup{spec}(f):=\textup{Im}(X)$. This is very similar to the fact that the Weyl group of $U(n)$ acts by permutations 
	on the eigenvalues of (the diagonalized form of) any skew-Hermitian matrix. For instance, when $\mathcal{E}=\mathcal{B}(n)$ 
	and $N=\mathbb{P}_{1}(\tfrac{1}{n})\cong S^{2}$ ($2$-sphere), then K\"{a}hler functions on $S^{2}$ are of the form 
	$f(x)=\langle u,x\rangle+c $, where $u\in \mathbb{R}^{3}$, $c\in \mathbb{R}$ and $\langle\,,\,\rangle$ is the Euclidean product 
	on $\mathbb{R}^{3}$. If $c=0$ and the Euclidean norm of $u$ is $j=n/2$, then 
	the spectrum (in our sense) of the function $f:S^{2}\to \mathbb{R}$, $x\mapsto \langle u,x\rangle$ is 
	$\textup{spec}(f)=\{-j,-j+1,...,j-1,j\}$. We recognize the spectrum of the usual spin operator of Quantum Mechanics 
	in the direction $u$. In \cite{Molitor-exponential}, 
	we describe in detail the connexion between this example and the unitary representations of $\mathfrak{su}(2)$ (used by physicists 
	to describe the spin of a particle), and conclude that the spin can be entirely understood mathematically from the study of 
	the Binomial distributions. This example shows that the techniques developped in this article 
	are useful to shed new light on some quantum mechanical problems.

	The paper is organized as follows. In Section \ref{nknkneknknfwknk}, we give a fairly detailed discussion on 
	the concept of torification, with which not all readers may be familiar. This includes: Dombrowski's construction, 
	parallel and fundamental lattices, lifting procedure and examples from Information Geometry. In Section \ref{ndknkkfeneknkndk}, 
	we study the group of diffeomorphisms preserving a parallel lattice, and prove a key technical result 
	(Proposition \ref{nfekwwndkefnk}). In Section \ref{nekneknfknknf}, we show that the normalizer $N(S)$ (see our discussion 
	above) is a semidirect product (Theorem \ref{nfkwndkefnknkwn}). 
	We deduce from this, in Section \ref{weeel}, our main result (Theorem \ref{nefknkenknkwndk}). 
	In Section \ref{nknkeknkenfkn}, we specialize to the case in which the dually flat manifold 
	is an exponential family. 

	\textbf{Notation.} Let $M$ and $M'$ be manifolds and $f:M\to M'$ a smooth map. The derivative of $f$ 
	is denoted by $f_{*}:TM\to TM'$. If $p\in M$, we write $f_{*_{p}}:T_{p}M\to T_{f(p)}M'$. Let $G$ be a Lie group 
	and $\Phi:G\times M\to M$ a Lie group action. Given $g\in G$, the map $\Phi_{g}$ is the diffeomorphism of $M$ 
	defined by $\Phi_{g}(p)=\Phi(g,p)$, where $p\in M$.

\section{Symplectic preliminaries: toric manifolds}\label{neknknekwnefknk}

	In this section we recall the rudiments of symplectic geometry used throughout the paper, with emphasis 
	on Hamiltonian torus actions. General references on symplectic geometry are 
	\cite{Ana-lectures,Ortega}. For torus actions, we recommend 
	\cite{Audin,Ana}, the historical papers \cite{Atiyah,Delzant,Guillemin82} and, in
	the K\"{a}hler case, \cite{Abreu,Victor}. 

	Let $G$ be a Lie group with Lie algebra $\textup{Lie}(G)=\mathfrak{g}$. Given $g\in G$, 
	we denote by $\textup{Ad}_{g}:\mathfrak{g}\to \mathfrak{g}$ and $\textup{Ad}_{g}^{*}:\mathfrak{g}^{*}\to 
	\mathfrak{g}^{*}$ the adjoint and coadjoint representations of $G$, respectively; they are related as follows: 
	$\langle \textup{Ad}_{g}^{*}\alpha, \xi\rangle=\langle \alpha,\textup{Ad}_{g^{-1}}\xi\rangle$, 
	where $\xi\in \mathfrak{g}$, $\alpha\in \mathfrak{g}^{*}$ (the dual of $\mathfrak{g}$) and $\langle\,,\,\rangle$ 
	is the natural pairing between $\mathfrak{g}$ and $\mathfrak{g}^{*}$. 

	Let $\Phi:G\times M\to M$ be a Lie group action of $G$ on a manifold $M$. The \textit{fundamental vector field} 
	associated to $\xi\in \mathfrak{g}$ is the vector field on $M$, denoted by $\xi_{M}$, defined by 
	$(\xi_{M})(p):=\tfrac{d}{dt}\big\vert_{0}\Phi(c(t),p),$ 
	where $p\in M$ and $c(t)$ is a smooth curve in $G$ satisfying $c(0)=e$ (neutral element) and $\dot{c}(0)=\xi$. 
	Given a map $\moment:M\to \mathfrak{g}^{*}$ and a vector $\xi\in \mathfrak{g}$, we will denote by $\moment^{\xi}:
	M\to \mathbb{R}$ the function given by $\moment^{\xi}(p):=\langle \moment(p),\xi\rangle.$
	We shall say that $\moment$ is \textit{$G$-equivariant} if for every $g\in G$, 
	$\moment \circ \Phi_{g}=\textup{Ad}_{g}^{*}\circ \moment.$ 
	Given a symplectic form $\omega$ on $M$, we say that $\Phi$ is \textit{symplectic} 
	if $(\Phi_{g})^{*}\omega=\omega$ for all $g\in G$. 

	Let $(M,\omega)$ be a symplectic manifold. A symplectic action $\Phi:G\times M\to M$ is said to be 
	\textit{Hamiltonian} if there exists a $G$-equivariant map $\moment:M\to \mathfrak{g}^{*}$, called 
	\textit{momentum map}, such that $\omega(\xi_{M},\,.\,)=d\moment^{\xi}(.)$ 
	(equality of 1-forms) for all $\xi\in \mathfrak{g}$. If the first de Rham cohomology group 
	$H^{1}(M,\mathbb{R})$ is trivial and 
	$G$ is compact, then any symplectic action of $G$ on $(M,\omega)$ is Hamiltonian 
	(see, e.g., \cite{Ortega}, Propositions 4.5.17 and 4.5.19).
	If $K$ is a closed subgroup of $G$, then the induced action 
	$K\times M\to M$ is also Hamiltonian. The momentum map $\moment$ has many nice properties (see, e.g., \cite{Ortega}). 
	For instance, $\moment$ is a submersion at $p\in M$ if and only if 
	the Hamiltonian action is locally free at $p$, that is, if the stabilizer 
	$G_{p}=\{g\in G\,\,\vert\,\,\Phi(g,p)=p\}$ is discrete. Thus, if $G$ acts freely on a $G$-invariant open 
	subset $U\subseteq M$, then $\moment$ is a submersion on $U$, which implies in particular that 
	$\moment(U)$ is open in $\mathfrak{g}^{*}$.

	When $G=\mathbb{T}^{n}=\mathbb{R}^{n}/\mathbb{Z}^{n}$ is a torus, it is convenient to identify 
	the Lie algebra of the torus $\mathbb{T}^{n}$ with $\mathbb{R}^{n}$ via the derivative at 
	$0\in \mathbb{R}^{n}$ of the quotient map $\mathbb{R}^{n}\to \mathbb{R}^{n}/\mathbb{Z}^{n}$, and 
	to identify $\mathbb{R}^{n}$ and its dual $(\mathbb{R}^{n})^{*}$ via the Euclidean metric.
	Upon these identifications, a momentum map for a Hamiltonian torus action 
	$\mathbb{T}^{n}\times M\to M$ can be regarded as a map $\moment:M\to \mathbb{R}^{n}$. 
	Moreover, since the coadjoint action of a commutative group is trivial, 
	the equivariance condition reduces to $\moment \circ \Phi_{g}=\moment$ for all $g\in \mathbb{T}^{n}$. 

	Suppose now that $\Phi:\mathbb{T}^{n}\times M\to M$ is a Hamiltonian action of the 
	torus $\mathbb{T}^{n}$ on a compact and connected 
	symplectic manifold $(M,\omega)$, with momentum map $\moment:M\to \mathbb{R}^{n}$. Let $M^{\circ}$ be the set 
	of points $p\in M$ where the torus action is free. The following properties are well-known. 
	\begin{description}
		\item[(P1)] $\Delta:=\moment(M)\subset \mathbb{R}^{n}$ is a convex polytope, called \textit{momentum polytope}. 
		\item[(P2)] If $\Phi$ is effective, then $\textup{dim}(M)\geq 2n$.
		\item[(P3)] If $\Phi$ is effective, then $M^{\circ}$ is a dense and connected open subset of $M$. 
	\end{description}
	The first item follows from the Atiyah-Guillemin-Sternberg convexity theorem \cite{Atiyah,Guillemin82}. 
	For the second and third item, see \cite[Theorem 27.3]{Ana} and \cite[Corollary B.48]{Guillemin},
	respectively. It is worth noting that the topological interior $\Delta^{\circ}$ of the momentum polytope
	is nonempty when $\Phi$ is effective. This follows from $\textup{(P3)}$ and the fact that 
	$\moment$ is a submersion on $M^{\circ}$. 

	If, in addition to the properties above, $\Phi$ is effective and $\textup{dim}(M)=2n$, then the
	quadruplet $(M,\omega,\Phi,\moment)$ is called a \textit{symplectic toric manifold}. In this case, 
	Delzant proved that the momentum polytope $\Delta$ completely determines $M:$
	\begin{description}
		\item[(P4)] If $(M',\omega',\Phi',\moment')$ is a symplectic 
			toric manifold such that $\moment'(M')=\moment(M)=\Delta$, 
			then there is an equivariant symplectomorphism $\varphi:M\to M'$ 
			satisfying $\moment=\moment'\circ \varphi$. 
	\end{description}
	Delzant also proved, among many other things, the following results (see \cite{Delzant}).
	\begin{description}
		\item[(P5)] $M$ is simply connected. 
		\item[(P6)] For every $p\in M$, the stabilizer
			$\textup{Stab}(p):=\{g\in \mathbb{T}^{n}\,\,\vert\,\,\Phi(g,p)=p\}$ is connected.
		\item[(P7)] For every $p\in M$, the rank of $\moment_{*_{p}}$ 
			is the dimension of the face of $\Delta$ whose relative interior\footnote{For an introduction to 
			convexity theory, we refer the reader the \cite{Rockafellar}.}
			contains $\moment(p)$.  
	\end{description}
	The momentum polytope itself is a face of $\Delta$ of dimension $n$ 
	whose relative interior is $\Delta^{\circ}$. Therefore 
	if $x=\moment(p)$ is in $\Delta^{\circ}$, then the rank of $\moment_{*_{p}}$ is $n$ by (P7). Thus $\moment$ 
	is a submersion at $p$, which means that $\textup{Stab}(p)$ is discrete. The later group is connected by 
	(P6) and hence it must be trivial. Thus $p\in M^{\circ}$. It follows that $\Delta^{\circ}\subset \moment(M^{\circ})$. 
	Combining this with our discussion above, we deduce that 

	\begin{description}
	\item[(P8)] $\moment(M^{\circ})=\Delta^{\circ}$. 
	\end{description}

	We now focus our attention on the K\"{a}hler case. 
	A symplectic toric manifold $(M,\omega,\Phi,\moment)$ is called a \textit{K\"{a}hler toric manifold} 
	if $M$ is a K\"{a}hler manifold whose K\"{a}hler form coincides with $\omega$, and if the torus acts by holomorphic and 
	isometric transformations on $M.$ In this case, there is a unique Riemannian metric $k$ on $\Delta^{\circ}$ that turns 
	$\moment:M^{\circ}\to \Delta^{\circ}$ into a Riemannian submersion. It can be proved that $k$ 
	is the Hessian of some potential $\varphi:\Delta^{\circ}\to \mathbb{R}$. Abreu 
	classified all potentials on $\Delta^{\circ}$ that are induced from compatible K\"{a}hler metrics on $M$ \cite{Abreu} 
	(see also \cite{Victor}).

	Let $J:TM\to TM$ be the complex structure of $M$ ($J\circ J=-Id$). 
	Given a vector field $X$ on $M$, we will denote by $\varphi^{X}_{t}$ the corresponding flow of $M$. 
	Thus $\tfrac{d}{dt}\varphi^{X}_{t}(p)= X(\varphi^{X}_{t}(p))$. Recall that 
	$x_{M}$ denotes the fundamental vector field of $x\in \mathbb{R}^{n}=\textup{Lie}(\mathbb{T}^{n})$ on $M$.
	If $x,y\in \mathbb{R}^{n}=\textup{Lie}(\mathbb{T}^{n})$, then $-Jx_{M}+y_{M}$ is a complete vector field on $M$ 
	(since $M$ is compact) and hence 
	$\varphi_{1}^{-Jx_{M}+y_{M}}(p)$ is well-defined for all $p\in M$. We define an action $\Phi^{\mathbb{C}}$ 
	of the algebraic torus $(\mathbb{C}^{*})^{n}$ on $M$ by 
	\begin{eqnarray}\label{nkenkenkenkfnekn}
		\Phi^{\mathbb{C}}\big((e^{2\pi z_{1}},...,e^{2\pi z_{n}}),p\big)=
			\varphi_{1}^{-Jx_{M}+y_{M}}(p)=\varphi^{-Jx_{M}}_{1}(\Phi([y],p)),
	\end{eqnarray}
	where $(z_{1},...,z_{n})\in \mathbb{C}^{n}$, $z_{k}=x_{k}+iy_{k}$, $x_{k},y_{k}\in \mathbb{R}$, and $k\in \{1,...,n\}$. 
	To see that \eqref{nkenkenkenkfnekn} is a group action, one may use the following facts: (1) 
	if $X$ and $Y$ are complete commuting vector fields, then 
	$\varphi_{t}^{X}\circ \varphi_{t}^{Y}=\varphi^{X+Y}_{t}$ for all $t\in \mathbb{R}$ and $(2)$ 
	the flow of a vector field $X$ on $M$ consists of holomorphic transformations if and only if 
	$[X,JY]=J[X,Y]$ for all vector fields $Y$ on $M$ (see \cite{moroianu_2007}, Lemma 2.7). 
	Obviously, $\Phi^{\mathbb{C}}$ is an extension of $\Phi$, provided $\mathbb{T}^{n}$ is identified with 
	$\{(e^{i2\pi t_{1}},...,e^{i2\pi t_{n}})\in \mathbb{C}^{n}\,\,\vert\,\,t_{1},...,t_{n}\in \mathbb{R}\}\subset \mathbb{C}^{n}$.
	We shall call it the \textit{holomorphic extension of} $\Phi$. 
	It can be proved that (see, e.g., \cite{molitor-toric}):
	\begin{description}
		\item[(P9)] $\Phi^{\mathbb{C}}$ is holomorphic as a map from $(\mathbb{C}^{*})^{n}\times M$ to $M$. 
		\item[(P10)] Given $p\in M^{\circ}$, the map $(\mathbb{C}^{*})^{n}\to M^{\circ}$, $z\mapsto \Phi^{\mathbb{C}}(z,p)$ 
			is a biholomorphism.
	\end{description}
	Fix $p_{0}\in M^{\circ}$ and identify $(\mathbb{C}^{*})^{n}$ with $M^{\circ}$ via the biholomorphism in $\textup{(P10)}$. 
	Consider the map $\sigma:(\mathbb{C}^{*})^{n}=M^{\circ}\to \mathbb{R}^{n}$ defined by 
	$\sigma(z_{1},...,z_{n})=\big(\tfrac{\ln(|z_{1}|)}{2\pi},...,\tfrac{\ln(|z_{n}|)}{2\pi}\big)$. There is a unique Riemannian 
	metric $h$ on $\mathbb{R}^{n}$ that turns $\sigma$ into a Riemannian submersion. 
	Let $(x_{1},...,x_{n})$ denote the usual linear coordinates 
	on $\mathbb{R}^{n}$. It can be proved that there exists a smooth function 
	$\psi:\mathbb{R}^{n}\to \mathbb{R}$ whose Hessian $(\tfrac{\partial^{2}\psi}{\partial x_{i}\partial x_{j}})$ is 
	$(h_{ij})=\big(h(\tfrac{\partial}{\partial x_{i}},\tfrac{\partial}{\partial x_{j}})\big)$ and 
	such that $-\textup{grad}(\psi)=-(\tfrac{\partial \psi}{\partial x_{1}},...,
	\tfrac{\partial\psi}{\partial x_{n}})$ is an isometry 
	from $(\mathbb{R}^{n},h)$ onto $(\Delta^{\circ},k)$, making the following diagram commutative:
	
	\begin{center}
	\begin{tikzcd}[column sep=small]
	& M^{\circ} \arrow[dl,"\displaystyle\sigma" above left] \arrow[dr, "\displaystyle\moment"] & \\
	(\mathbb{R}^{n},h) \arrow["-\textup{grad}(\psi)" below]{rr} & & (\Delta^{\circ},k)
	\end{tikzcd}
	\end{center}

	When the K\"{a}hler metric on $M$ is real analytic, it can be shown that $M$, regarded as a 
	K\"{a}hler toric manifold, is completely determined by either $(\mathbb{R}^{n},h)$ or $(\Delta^{\circ},k)$, up 
	to an equivariant K\"{a}hler isomorphism and reparametrization of the torus 
	(see \cite{molitor-toric} and Theorem \ref{newdnkekfwndknk} below). This result 
	relies on the crucial observation that $(\mathbb{R}^{n},h)$ and $(\Delta^{\circ},k)$, 
	besides being Riemannian manifolds, are also 
	affine manifolds (since they are open subsets of $\mathbb{R}^{n}$), with good analytic properties, namely they are dually 
	flat (see Section \ref{nknkneknknfwknk}). 
	In the next section, we delve deeper into the correspondence between toric K\"{a}hler manifolds 
	and dually flat manifolds, using the unifying language of ``torification", that we introduced in \cite{molitor-toric}. 


\section{Torification of dually flat manifolds}\label{nknkneknknfwknk}
	In this section, we discuss the concept of torification, which is used throughout this paper. This concept 
	is a combination of two ingredients: (1) \textit{Dombrowski's construction}, 
	which implies that the tangent bundle of a dually flat manifold is naturally a K\"{a}hler manifold \cite{Dombrowski}, 
	and (2) \textit{parallel lattices}, which are used to implement torus actions. Further properties are discussed as well, 
	such as the lifting property mentioned in the introduction. Examples from Information Geometry are presented. 
	The material is mostly taken from \cite{molitor-toric}.

\subsection{Dombrowski's construction.}\label{nkwwnkwnknknfkn}
	Let $M$	be a connected manifold of dimension $n$, endowed with a Riemannian metric $h$ 
	and affine connection $\nabla$ ($\nabla$ is not necessarily 
	the Levi-Civita connection). The \textit{dual connection} of $\nabla$, denoted by $\nabla^{*}$, is the only connection satisfying 
	$X(h(Y,Z))=h(\nabla_{X}Y,Z)+h(Y,\nabla^{*}_{X}Z)$ for all vector fields $X,Y,Z$ on $M$. 
	When both $\nabla$ and $\nabla^{*}$ are flat (i.e., the curvature tensor and torsion are zero), 
	we say that the triple $(M,h,\nabla)$ is a \textit{dually flat manifold}. 

	Let $\pi:TM\to M$ denote the canonical projection. Given a local coordinate system $(x_{1},...,x_{n})$ 
	on $U\subseteq M$, we can define a coordinate system $(q,r)=(q_{1},...,q_{n},r_{1},...,r_{n})$ 
	on $\pi^{-1}(U)\subseteq TM$ by $(q,r)(\sum_{j=1}^{n}a_{j}\tfrac{\partial}{\partial x_{j}}\big\vert_{p})=
	(x_{1}(p),...,x_{n}(p),a_{1},...,a_{n})$, where $p\in M$ and $a_{1},...,a_{n}\in \mathbb{R}$. 
	Write $(z_{1},...,z_{n})=(q_{1}+ir_{1},...,q_{n}+ir_{n})$, where $i=\sqrt{-1}$. When $\nabla$ is flat, Dombrowski 
	\cite{Dombrowski} showed that the family of complex coordinate systems $(z_{1},...,z_{n})$ 
	on $TM$ (obtained from affine coordinates on $M$) form a holomorphic atlas on $TM$. Thus, when $\nabla$ is flat, 
	$TM$ is naturally a complex manifold. If in addition $\nabla^{*}$ is flat, then $TM$ has a natural 
	K\"{a}hler metric $g$ whose local expression in the coordinates $(q,r)$ is given by $g(q,r)=
	\big[\begin{smallmatrix}
		h(x)  &  0\\
		0     &   h(x)
	\end{smallmatrix}
	\big]$, where $h(x)$ is the matrix representation of $h$ in the affine coordinates $x=(x_{1},...,x_{n})$. 
	It follows that the tangent bundle of a dually flat manifold is naturally a K\"{a}hler manifold. In this paper, 
	we will refer to this K\"{a}hler structure as the \textit{K\"{a}hler structure associated to Dombrowski's 
	construction}. 

	Dombrowski's construction can be described in a coordinate-free fashion by means of \textit{connectors} 
	(the concept was introduced by Dombrowski \cite{Dombrowski}). Connectors can be defined as follows.
	Let $u=\sum_{k=1}^{n}u_{k}\tfrac{\partial}{\partial x_{k}}\big\vert_{p}\in \pi^{-1}(U)$ be arbitrary. 
	Define a linear map $K_{u}:T_{u}(TM)\to T_{p}M$ by 
	$K_{u}\big(\tfrac{\partial}{\partial q_{a}}\big\vert_{u}\big):=
	\sum_{k,j=1}^{n}\Gamma_{aj}^{k}(p)u_{j}\tfrac{\partial}{\partial x_{k}}\big\vert_{p}$
	and $K_{u}\big(\tfrac{\partial}{\partial r_{a}}\big\vert_{u}\big):=
	\tfrac{\partial}{\partial x_{a}}\big\vert_{p}$ for all $a\in \{1,...,n\}$, 
	where the $\Gamma_{ij}^{k}$'s are the Christoffel symbols of $\nabla$. If $X$ and $Y$ are 
	vector fields on $M$ such that $Y(p)=u$, then a direct computation shows that $K_{u}(Y_{*_{p}}X_{p})=(\nabla_{X}Y)(p)$.
	Since vectors of the form $Y_{*_{p}}X_{p}$, with $Y_{p}=u$, generate $T_{u}(TM)$, it follows that 
	the definition of $K_{u}$ is independant of the choice of the chart $(U,\varphi)$. The map 
	$K:TTM\to TM$, defined for $A\in T_{u}(TM)$ by $K(A):=K_{u}(A)$, is called \textit{connector}, or \textit{connection map}, 
	associated to $\nabla$. 
	
	Given $u\in T_{p}M$, the map $T_{u}(TM)\to T_{p}M\oplus T_{p}M$, defined 
	by $A\mapsto(\pi_{*_{u}}A,KA)$, is easily seen to be a linear bijection. Therefore, at any point $u\in T_{p}M$ 
	we can identify the vector spaces  $T_{u}(TM)$ and $T_{p}M\oplus T_{p}M$. In terms of this identification, the K\"{a}hler metric 
	$g$ and complex structure $J$ on $TM$ read 
	$g_{u}((v,w), (v', w'))
	=h_{p}(v,v')+ h_{p}(w,w')$ and 
	$J_{u}((v,w))= (-w,v)$, where $u,v,w,v',w'\;\in\; T_{p}M$.

\subsection{Parallel lattices.}\label{nekwnkefkwnkk}

	Let $(M,h,\nabla)$ be a dually flat manifold of dimension $n$. 
	A subset $L\subset TM$ is said to be a \textit{parallel lattice} with respect to $\nabla$ if there are $n$ vector 
	fields $X_{1},...,X_{n}$ on $M$ that are parallel with respect to $\nabla$ and such that: $(i)$ $\{X_{1}(p),...,X_{n}(p)\}$ is a basis 
	for $T_{p}M$ for every $p\in M$, and $(ii)$ $L=\{k_{1}X_{1}(p)+...+k_{n}X_{n}(p)\,\,\vert\,\,k_{1},...,k_{n}\in \mathbb{Z},\,\,p\in M\}$. 
	In this case, we say that the frame $X=(X_{1},...,X_{n})$ is a 
	\textit{generator} for $L$. We will denote the set of generators for $L$ by 
	$\textup{gen}(L)$. 

	Given a parallel lattice $L\subset TM$ with respect to $\nabla$, and $X\in \textup{gen}(L)$ , 
	we will denote by $\Gamma(L)$ the set of 
	transformations of $TM$ of the form $u\mapsto u+k_{1}X_{1}+...+k_{n}X_{n}$, where 
	$u\in TM$ and $k_{1},...,k_{n}\in \mathbb{Z}$. The group
	$\Gamma(L)$ is independent of the choice of $X\in \textup{gen}(L)$ and is isomorphic to $\mathbb{Z}^{n}$. 
	Moreover, the natural action of $\Gamma(L)$ 
	on $TM$ is free and proper. Thus the quotient $M_{L}=TM/\Gamma(L)$ is a 
	smooth manifold and the quotient map $q_{L}:TM\to M_{L}$ is a covering map 
	whose Deck transformation group is $\Gamma(L)$. Since 
	$\pi\circ \gamma=\pi$ for all $\gamma\in \Gamma(L)$, 
	there exists a surjective submersion $\pi_{L}:M_{L}\to M$ such that $\pi=\pi_{L}\circ q_{L}$.
%

	Let $\mathbb{T}^{n}=\mathbb{R}^{n}/\mathbb{Z}^{n}$ be the $n$-dimensional torus. 
	Given $t=(t_{1},...,t_{n})\in \mathbb{R}^{n}$, we will denote 
	by $[t]=[t_{1},...,t_{n}]$ the corresponding equivalence class in $\mathbb{R}^{n}/\mathbb{Z}^{n}$. 
	Given $X\in \textup{gen}(L)$, we will denote by 
	$\Phi_{X}:\mathbb{T}^{n}\times M_{L}\to M_{L}$ the effective torus action defined by 
	\begin{eqnarray}\label{nkdnknewknek}
		\Phi_{X}([t],q_{L}(u))=q_{L}(u+t_{1}X_{1}+...+t_{n}X_{n}). 
	\end{eqnarray}

	The manifold $M_{L}=TM/\Gamma(L)$ is naturally a K\"{a}hler manifold (this follows from the fact that 
	every $\gamma\in \Gamma(L)$ is a holomorphic and isometric map with respect the K\"{a}hler structure 
	associated to Dombrowski's construction). Moreover, for each $a\in \mathbb{T}^{n}$, the map 
	$(\Phi_{X})_{a}:M_{L}\to M_{L}$, $p\mapsto \Phi_{X}(a,p)$ is a holomorphic isometry. 

\subsection{Torification.}\label{nfkenkfejdefdknkn} 
	Let $(M,h,\nabla)$ be a connected dually flat manifold of dimension $n$ and $N$ a connected K\"{a}hler 
	manifold of complex dimension $n$, equipped with an effective holomorphic and isometric torus 
	action $\Phi:\mathbb{T}^{n}\times N\to N$. Let $N^{\circ}$ denote the set of points $p\in N$ 
	where $\Phi$ is free. 

\begin{definition}[\textbf{Torification}]
	We shall say that $N$ is a \textit{torification} of $M$ 
	if there exist a parallel lattice $L\subset TM$ with respect to 
	$\nabla$, $X\in \textup{gen}(L)$ and a holomorphic and isometric diffeomorphism $F:M_{L}\to N^{\circ}$
	satisfying $F\circ (\Phi_{X})_{a}=\Phi_{a}\circ F$ for all $a\in \mathbb{T}^{n}$. 
\end{definition}

	By abuse of language, we will often say that the torus action $\Phi:\mathbb{T}^{n}\times N\to N$ is a torification of $M$. 
	We shall say that a K\"{a}hler manifold $N$ is \textit{regular} if it is connected, 
	simply connected, complete and if the K\"{a}hler metric is real analytic. 
	A torification $\Phi:\mathbb{T}^{n}\times N\to N$ is said to be \textit{regular} if $N$ is regular. 
	In this paper, we are mostly interested 
	in regular torifications. 
\begin{remark}\label{ndkdnknknknk}
	\textbf{}
	\begin{enumerate}[(a)]
		\item A torification $N$ is not necessarily a K\"{a}hler toric manifold. For example, $N$ may not be compact 
			(see Proposition \ref{nfeknwknwknk}). 
		\item Since a toric K\"{a}hler manifold is simply connected by $\textup{(P5)}$, it is 
			regular if and only if the K\"{a}hler metric is real analytic. 
		\item If a torification $\Phi:\mathbb{T}^{n}\times N\to N$ is compact and regular, then $\Phi$ is Hamiltonian 
			(since $N$ is simply connected) and hence $N$ is a toric K\"{a}hler manifold.
	\end{enumerate}
\end{remark}

	Two torifications $\Phi:\mathbb{T}^{n}\times N\to N$ and $\Phi':\mathbb{T}^{n}\times N'\to N'$ 
	of the same connected dually flat manifold $(M,h,\nabla)$ are said to be \textit{equivalent} if there exists 
	a K\"{a}hler isomorphism $f:N\to N'$ and a Lie group isomorphism $\rho:\mathbb{T}^{n}\to \mathbb{T}^{n}$ such that 
	$f\circ \Phi_{a}=\Phi'_{\rho(a)}\circ f$ for all $a\in \mathbb{T}^{n}$.
\begin{theorem}[\textbf{Equivalence of regular torifications}]\label{newdnkekfwndknk}
	Regular torifications of a connected dually flat manifold $(M,h,\nabla)$ are equivalent. 
\end{theorem}

	A connected dually flat manifold $(M,h,\nabla)$ is said to be \textit{toric} if it has a regular torification $N$. 
	In this case, we will often refer to $N$ as ``the regular torification of $M$", and keep in mind that it is 
	only defined up to an equivariant K\"{a}hler isomorphism and reparametrization of the torus. 

	For later reference, we give the following technical definition.
\begin{definition}
	Suppose $\Phi:\mathbb{T}^{n}\times N\to N$ is a torification of $(M,h,\nabla)$. 
	\begin{enumerate}[(1)]
	\item A \textit{toric parametrization} is a triple $(L,X,F)$, where $L\subset TM$ is parallel lattice with respect to $\nabla$, 
		$X\in \textup{gen}(L)$ and $F:M_{L}\to N^{\circ}$ is a holomorphic and isometric diffeomorphism satisfying 
		$F\circ (\Phi_{X})_{a}=\Phi_{a}\circ F$ for all $a\in \mathbb{T}^{n}$.
	\item  Let $\tau:TM\to N^{\circ}$ and $\kappa:N^{\circ}\to M$ be smooth maps. We say that 
		the pair $(\tau,\kappa)$ is a \textit{toric factorization} 
		if there exists a toric parametrization $(L,X,F)$ such that $\tau=F\circ q_{L}$ and $\kappa=\pi_{L}\circ F^{-1}$. 
	In this case, we say that $(\tau, \kappa)$ is \textit{induced by the toric parametrization} $(L,X,F)$. Note that $\pi=\kappa\circ \tau.$
\item We say that $\kappa:N^{\circ}\to M$ is the \textit{compatible projection induced by the toric parametrization} $(L,X,F)$ 
	if there exists a map $\tau:TM\to N^{\circ}$ such that $(\tau,\kappa)$ is the toric factorization induced by $(L,X,F)$. 
	When it is not necessary to mention $(L,X,F)$ explicitly, we just say 
	that $\kappa$ is a \textit{compatible projection}. Analogously, one defines a \textit{compatible covering map} $\tau:TM\to N^{\circ}$. 
\end{enumerate}
\end{definition}
	\noindent By abuse of language, we will often say that the formula $\pi=\kappa\circ \tau$ is a toric factorization. 

	If $\Phi:\mathbb{T}^{n}\times N\to N$ is a regular torification of $M$, and 
	if $\kappa,\kappa':N^{\circ}\to M$ are compatible projections, 
	then there is a holomorphic isometry $\varphi:N\to N$ and a Lie group isomorphism 
	$\rho:\mathbb{T}^{n}\to \mathbb{T}^{n}$ such that $\varphi\circ \Phi_{a}=\Phi_{\rho(a)}\circ \varphi$ 
	for all $a\in \mathbb{T}^{n}$ and $\kappa'=\kappa\circ \varphi$.

\subsection{Lifting procedure.}\label{neknkenkfnkn}
	Let $\Phi:\mathbb{T}^{n}\times N\to N$ and $\Phi':\mathbb{T}^{d}\times N'\to N'$ be torifications of the dually flat manifolds 
	$(M,h,\nabla)$ and $(M',h',\nabla')$, respectively. 

\begin{definition}
	Let $f:M\to M'$ and $m:N\to N'$ be smooth maps. We say that $m$ is a \textit{lift of} $f$ 
	if there are compatible covering maps $\tau:TM\to N^{\circ}$ and $\tau':TM'\to (N')^{\circ}$ such that 
	$m\circ \tau= \tau'\circ f_{*}$. In this case, we say that $m$ \textit{is a lift of} $f$ 
	\textit{with respect to} $\tau$ and $\tau'$.  
\end{definition}

	If $m$ is a lift of $f$ with respect to $\tau$ and $\tau'$, and if $\pi=\kappa\circ \tau$ and $\pi'=\kappa'\circ \tau'$ are toric factorizations, then 
	$\kappa'\circ m =f\circ \kappa.$ Therefore the following diagram commutes:

\begin{eqnarray}\label{knkdfnknkndknk}
	\begin{tikzcd}
		TM \arrow{rrr}{\displaystyle f_{*}} \arrow[dd,swap,"\displaystyle\pi"] \arrow{rd}{\displaystyle \tau} &    &   & 
			TM' \arrow[swap]{ld}{\displaystyle \tau'} \arrow[dd,"\displaystyle\pi'"]  \\
		& N^{\circ} \arrow{r}{\displaystyle m} \arrow{ld}{\displaystyle \kappa}  
			& (N')^{\circ} \arrow[swap]{dr}{\displaystyle \kappa'}  & \\
		M \arrow[swap]{rrr}{\displaystyle f}  &     && M'
	\end{tikzcd}
\end{eqnarray}

	If $m:N\to N'$ is a lift of $f$, then $m$ is a K\"{a}hler immersion if and only if $f:(M,\nabla)\to (M',\nabla')$ is an affine immersion satisfying $f^{*}h'=h$. 
	In this case, there exists a unique Lie group homomorphism 
	$\rho:\mathbb{T}^{n}\to \mathbb{T}^{d}$ with finite kernel such that $m\circ \Phi_{a}=\Phi_{\rho(a)}'\circ m$ for all $a\in \mathbb{T}^{n}$.

\begin{theorem}[\textbf{Existence of lifts}]\label{ncnkwknknk}
	Suppose $\Phi:\mathbb{T}^{n}\times N\to N$ and $\Phi':\mathbb{T}^{d}\times N'\to N'$ are regular torifications of $(M,h,\nabla)$ and $(M',h',\nabla')$, respectively. 
	Let $\tau:TM\to N^{\circ}$ and $\tau':TM'\to (N')^{\circ}$ be compatible covering maps. Then every isometric affine immersion 
	$f:M\to M'$ has a unique lift $m:N\to N'$ with respect to $\tau$ and $\tau'$. 
\end{theorem}

	For later reference, we give the following technical result, in which $(L,X,F)$ and $(L',X',F')$ 
	are toric parametrizations of $N$ and $N'$, respectively, with induced toric factorizations 
	$\pi=\kappa\circ \tau:TM\to M$ and $\pi'=\kappa'\circ \tau':TM'\to M'$. 

\begin{proposition}[\textbf{The derivative of an isometric affine immersion is equivariant}]\label{nfekwnknfkenk}
	Let $m:N\to N'$ be the lift of an affine isometric immersion $f:M\to M'$ 
	with respect to $\tau$ and $\tau'$, respectively, and let $\rho:\mathbb{T}^{n}\to \mathbb{T}^{d}$ be the unique 
	homomorphism satisfying $m\circ \Phi_{a}=\Phi'_{\rho(a)}\circ m$ for all $a\in \mathbb{T}^{n}$. If 
	$N$ and $N'$ are regular, then the derivative of $f$ satisfies 
	\begin{eqnarray*}
		f_{*}\circ (T_{X})_{t}=(T_{X'})_{\rho_{*_{e}}t}\circ f_{*}
	\end{eqnarray*}
	for all $t\in\mathbb{R}^{n}\cong \textup{Lie}(\mathbb{T}^{n})$, where $T_{X}$ is the group action of 
	$\mathbb{R}^{n}$ on $TM$ defined by 
	$T_{X}(t,u)=u+\sum_{k=1}^{n}t_{k}X_{k}$ ($T_{X'}$ is defined 
	similarly). 
\end{proposition}

%
%
%
%
%

\subsection{Fundamental lattices.}\label{nfeknkfnekndk}

	Let $\Phi:\mathbb{T}^{n}\times N\to N$ and $\Phi':\mathbb{T}^{n}\times N'\to N'$ be regular torifications of the dually 
	flat manifold $(M,h,\nabla)$. 
	If $(L,X,F)$ and $(L',X',F')$ are toric parametrizations of $N$ and $N'$, respectively, then $L=L'$. 
	We call $\mathscr{L}:=L=L'$ the \textit{fundamental lattice} of $(M,h,\nabla)$. 

\subsection{Canonical example: K\"{a}hler toric manifolds.}\label{nknknvknknk} Let $(N,\omega,\Phi,\moment)$ be a K\"{a}hler toric manifold, 
	with K\"{a}hler metric $g$ and 
	momentum polytope $\Delta=\moment(N)\subset \mathbb{R}^{n}$. Let $k$ (resp. $h$) be the unique Riemannian metric on 
	$\Delta^{\circ}$ (resp. $\mathbb{R}^{n}$)
	that turns $\moment:(N^{\circ},g)\to (\Delta^{\circ},k)$ (resp. $\sigma:(N^{\circ},g)\to (\mathbb{R}^{n},h)$) 
	into a Riemannian submersion (see Section \ref{neknknekwnefknk}). 
	Let $\nabla^{\textup{flat}}$ be the flat connection on $\mathbb{R}^{n}$ and 
	let $\nabla^{k}$ be the dual connection 
	of $\nabla^{\textup{flat}}$ on $\Delta^{\circ}$ with respect to $k$. 
	Then $(\Delta^{\circ},k,\nabla^{k})$ and $(\mathbb{R}^{n},h,\nabla^{\textup{flat}})$ are isomorphic
	dually flat manifolds and $\Phi:\mathbb{T}^{n}\times N\to N$ 
	is simultaneously a torification of $(\Delta^{\circ},k,\nabla^{k})$ and $(\mathbb{R}^{n},h,\nabla^{\textup{flat}})$. 
	From now on we call $(\Delta^{\circ},k,\nabla^{k})$ the \textit{momentum polytope} and 
	$(\mathbb{R}^{n},h,\nabla^{\textup{flat}})$ the \textit{holomorphic polytope} of $N$. Note that the holomorphic 
	polytope depends on the choice of a point $p\in N^{\circ}$.

\subsection{Examples from Information Geometry.}\label{nknkefnknk} 

	The concept of torification was motivated, in the first place, by the 
	connection between K\"{a}hler geometry, Information Geometry and Quantum Mechanics.
	In this section, we illustrate this connection with a few examples. The reader interested in 
	Information Geometry may consult \cite{Jost2,Amari-Nagaoka,Murray}.

\begin{definition}\label{def:5.1}
	A \textit{statistical manifold} is a pair $(S,j)$, where $S$ is a manifold and where $j$ 
	is an injective map from $S$ to the space of all probability density functions $p$ 
	defined on a fixed measure space $(\Omega,dx)$: 
	\begin{eqnarray*}
		j:S\hookrightarrow\Bigl\{ p:\Omega\to\mathbb{R}\;\Bigl|\; p\:\textup{is measurable, }p\geq 
		0\textup{ and }\int_{\Omega}p(x)dx=1\Bigr\}.
	\end{eqnarray*}
\end{definition}

	If $\xi=(\xi_{1},...,\xi_{n})$ is a coordinate system on a statistical manifold $S$, then we shall 
	indistinctly write $p(x;\xi)$ or $p_{\xi}(x)$ for the probability density function determined by $\xi$.

	Given a ``reasonable" statistical manifold $S$, it is possible to define a metric $h_{F}$ and a 
	family of connections $\nabla^{(\alpha)}$ on $S$ $(\alpha\in\mathbb{R})$ in the following way: 
	for a chart $\xi=(\xi_{1},...,\xi_{n})$ of $S$, define
	\begin{alignat*}{1}
		\bigl(h_F\bigr)_{\xi}\bigl(\partial_{i},\partial_{j}\bigr) & :=
		\mathbb{E}_{p_{\xi}}\bigl(\partial_{i}\ln\bigl(p_{\xi}\bigr)\cdotp\partial_{j}\ln\bigl(p_{\xi}\bigr)\bigr),
		\nonumber\\
		\Gamma_{ij,k}^{(\alpha)}\bigl(\xi\bigr) & :=
		\mathbb{E}_{p_{\xi}}\bigl[\bigl(\partial_{i}\partial_{j}\ln\bigl(p_{\xi}\bigr)
		+\tfrac{1-\alpha}{2}\partial_{i}\ln\bigl(p_{\xi}\bigr)\cdotp\partial_{j}\ln\bigl(p_{\xi}\bigr)\bigr)
		\partial_{k}\ln\bigl(p_{\xi}\bigr)\bigr],\label{eq:48}\nonumber
	\end{alignat*}
	where $\mathbb{E}_{p_{\xi}}$ denotes the mean, or expectation, with respect to the probability 
	$p_{\xi}dx$, and where $\partial_{i}$ is a shorthand for $\tfrac{\partial}{\partial\xi_{i}}$. 
	In the formulas above, it is assumed that the function $p_{\xi}(x)$ is smooth with respect to 
	$\xi$ and that the expectations are finite. When the first formula above defines a smooth metric $h_{F}$ on $S$, 
	it is then called \textit{Fisher metric}. In this case, the 
	$\Gamma_{ij,k}^{(\alpha)}$'s define a connection $\nabla^{(\alpha)}$ on $S$ via the formula 
	$\Gamma_{ij,k}^{(\alpha)}(\xi)=(h_F)_{\xi}(\nabla_{\partial_{i}}^{(\alpha)}\partial_{j},\partial_{k})$, 
	which is called the \textit{$\alpha$-connection}.

%
	Among the $\alpha$-connections, the $1$-connection is particularly important and 
	is usually referred to as the \textit{exponential connection}, also denoted by $\nabla^{(e)}$. 
	In this paper, we will only consider statistical manifolds $S$ for which the Fisher metric $h_{F}$ and
	exponential connection $\nabla^{(e)}$ are well defined.


	We now recall the definition of an exponential family. 
\begin{definition}\label{def:5.3} 
	An \textit{exponential family} $\mathcal{E}$ on a measure space $(\Omega,dx)$ is a set of probability 
	density functions $p(x;\theta)$ of the form 
	$p(x;\theta)=\exp\big\{ C(x)+\sum_{i=1}^{n}\theta_{i}F_{i}(x)-\psi(\theta)\big\},$ 
	where $C,F_1...,F_n$ are measurable functions on $\Omega$, $\theta=(\theta_{1},...,\theta_{n})$ 
	is a vector varying in an open subset $\Theta$ of $\mathbb{R}^{n}$ and where $\psi$ is a function defined on $\Theta$.
\end{definition}
	In the above definition, it is assumed that the family of functions $\{1,F_1,...,$ $F_n\}$ 
	is linearly independent, so that the map $p(x,\theta)\mapsto\theta$ becomes a bijection, hence defining a global 
	chart for $\mathcal{E}$. The parameters $\theta_{1},...,\theta_{n}$ are called the 
	\textit{natural} or \textit{canonical parameters} of the exponential family $\mathcal{E}$. 
	
\begin{example}[\textbf{Binomial distribution}]\label{exa:5.6}
	Let $\mathcal{B}(n)$ be the set of Binomial distributions $p(k)=\binom{n}{k}q^{k}\bigl(1-q\bigr)^{n-k}$, $q\in (0,1)$, 
	defined over $\Omega:=\{0,...,n\}$. It is a 1-dimensional exponential family, because 
	$p(k)=\exp\big\{ C(k)+\theta F(k)-\psi(\theta)\big\}$, where $\theta=\ln\big(\frac{q}{1-q}\big)$, $C(k)= \ln\binom{n}{k}$, 
	$F(k)=k$, $k\in \Omega$, and $\psi(\theta)=n\ln\big(1+\exp(\theta)\big)$. 
\end{example}
\begin{example}[\textbf{Categorical distribution}]\label{exa:5.5}
	Let $\Omega=\{ x_{1},...,x_{n},...\}$ be a finite set and let $\mathcal{P}_{n}^{\times}$ be the set of maps
	$p:\Omega\to \mathbb{R}$ satisfying $p(x)>0$ for all $x\in \Omega$ and $\sum_{x\in \Omega}p(x)=1$. Then 
	$\mathcal{P}_{n}^{\times}$ is an exponential family of dimension $n-1$. Indeed, elements in $\mathcal{P}_{n}^{\times}$ 
	can be parametrized as follows: $p(x;\theta)=\exp\big\{\sum_{i=1}^{n-1}\theta_{i}F_{i}(x)-\psi(\theta)\big\}$,
	where $x\in \Omega$, $\theta=(\theta_{1},...,\theta_{n-1})\in\mathbb{R}^{n-1},$ 
	$F_{i}(x_{j})=\delta_{ij}$ \textup{(Kronecker delta)} and $\psi(\theta)=\ln\big(1+\sum_{i=1}^{n-1}e^{\theta_{i}}\big)$. 
%
%
%
%
\end{example}

\begin{example}[\textbf{Poisson distribution}]
	A Poisson distribution is a distribution over $\Omega=\mathbb{N}=\{0,1,...\}$ of the form 
	$p(k;\lambda)=e^{-\lambda}\tfrac{\lambda^{k}}{k!}$, 
	where $k\in \mathbb{N}$ and $\lambda>0$. Let $\mathscr{P}$ denote the set of all Poisson distributions 
	$p(\,.\,;\lambda)$, $\lambda>0$. The set $\mathscr{P}$ is an exponential family, because 
	$p(k,\lambda)=\textup{exp}\big(C(k)+F(k)\theta-\psi(\theta)\big),$ where 
	$C(k)=-\ln(k!)$, $F(k)=k$, $\theta=\ln(\lambda)$, and $\psi(\theta)=\lambda=e^{\theta}.$
\end{example}

	Under mild assumtions, it can be proved that an exponential family $\mathcal{E}$ endowed with 
	the Fisher metric $h_{F}$ and exponential connection $\nabla^{(e)}$ is a dually flat manifold 
	(see \cite{Amari-Nagaoka}). This is the case, for example, if $\Omega$ is a finite set endowed 
	with the counting measure (see \cite{shima}, Chapter 6). In the sequel, we will always regard 
	an exponential family  $\mathcal{E}$ as a dually flat manifold. 

	Since an exponential family is dually flat, it is natural to ask whether it is toric. 
	Below we describe three examples. Let $\mathbb{P}_{n}(c)$ be the complex projective space of complex dimension $n$, 
	endowed with the Fubini-Study metric normalized 
	in such a way that the holomorphic sectional curvature is $c>0$. Let $\Phi_{n}$ be the action of 
	$\mathbb{T}^{n}$ on $\mathbb{P}_{n}(c)$ defined by 
	\begin{eqnarray*}
		\Phi_{n}([t],[z])= [e^{2i\pi t_{1}}z_{1},...,e^{2i\pi t_{n}}z_{n},z_{n+1}]. \,\,\,\,\,\,\,(\textup{homogeneous coordinates})
	\end{eqnarray*}
\begin{proposition}\label{nfeknwknwknk}
	In each case below, the torus action is a regular torification of the indicated exponential family $\mathcal{E}$. \\

	\begin{tabular}{lll}
		$(a)$ &  $\mathcal{E}=\mathcal{P}_{n+1}^{\times}$,    &    
						$\Phi_{n}:\mathbb{T}^{n}\times \mathbb{P}_{n}(1)\to \mathbb{P}_{n}(1)$.\\[0.4em]
		$(b)$ & 
	$\mathcal{E}=\mathcal{B}(n)$,                &    $\Phi_{1}:\mathbb{T}^{1}\times 
								\mathbb{P}_{1}(\tfrac{1}{n})\to \mathbb{P}_{1}(\tfrac{1}{n})$.\\[0.4em]
		$(c)$ & $\mathcal{E}=\mathscr{P}$,                   & $\mathbb{T}^{1}\times \mathbb{C}\to \mathbb{C}$, $([t],z)\mapsto e^{2i\pi t}z$.	
	\end{tabular}
\end{proposition}

	\noindent For a proof and more examples, see \cite{molitor-toric}.

\section{Diffeomorphisms preserving a parallel lattice} \label{ndknkkfeneknkndk}

	Let $(M,h,\nabla)$ be a connected dually flat manifold of dimension $n$ (not necessarily toric), and 
	suppose $L\subset TM$ is a parallel lattice with respect to $\nabla$, generated by $X=(X_{1},...,X_{n})$. 

	The objective of this section is to show Proposition \ref{nfekwwndkefnk} below, which is a key technical result. 
	We will adopt the notation of Section \ref{nkwwnkwnknknfkn}. 
	Thus $M_{L}$ is the quotient manifold $TM/\Gamma(L)\cong TM/\mathbb{Z}^{n}$, $q_{L}:TM\to M_{L}$ is the corresponding 
	quotient map and $\Phi=\Phi_{X}$ is the torus action on $M_{L}$ associated to $X$.

\begin{definition}
	A diffeomorphism $\psi:M\to M$ \textit{preserves} the parallel lattice $L$ if $\psi_{*}(L)= L$.
\end{definition}

	Note that the set of all diffeomorphisms of $M$ preserving a parallel lattice $L$ is a group, which we denote by 
	$\textup{Diff}(M,L)$. Below, we will mostly focus on diffeomorphisms of $\textup{Diff}(M,L)$ that are isometries. 
	We will use the following notation:
	\begin{itemize}
		\item $\textup{Isom}(M,h)$ is the group of isometries of $(M,h)$.
		\item $\textup{Diff}(M,\nabla)$ is the group of diffeomorphisms of $M$ that are affine with respect 
			to $\nabla$.
		\item $\textup{Diff}(M,h,L)=\textup{Diff}(M,L)\cap \textup{Isom}(M,h)$.
		\item $\textup{Diff}(M,h,\nabla)=\textup{Diff}(M,\nabla)\cap \textup{Isom}(M,h)$.  
	\end{itemize}

	Let $\textup{GL}(n,\mathbb{Z})$ be the group of $n\times n$ matrices with integer entries. 
	The following lemma is immediate.

\begin{lemma}	\label{nfekndkwenknfk}
	Let $\psi:M\to M$ be a diffeomorphism. Then $\psi$ preserves $L$ if and only if 
	there exists a matrix $R=(r_{ij})\in \textup{GL}(n,\mathbb{Z})$ such that 
	$\psi_{*_{p}}X_{k}(p)=\sum_{i=1}^{n}r_{ik}X_{i}(\psi(p))$
	for all $p\in M$ and all $k\in \{1,...,n\}$. 
\end{lemma}

	A simple consequence of Lemma \ref{nfekndkwenknfk} is the following 

\begin{lemma}\label{nnknefkwnknefknk}
	$\textup{Diff}(M,L)\subseteq \textup{Diff}(M,\nabla)$.
\end{lemma}
\begin{proof}
	Let $\psi\in \textup{Diff}(M,L)$ be arbitrary. By hypothesis, $X=(X_{1},...,X_{n})$ is a generator for $L$, and so
	each vector field $X_{i}$ is parallel with respect to $\nabla$, which implies that 
	$[X_{i},X_{j}]=\nabla_{X_{i}}X_{j}-\nabla_{X_{j}}X_{i}=0$ for all $i,j\in \{1,...,n\}$. Thus, for a given $p\in M$, 
	there are affine coordinate systems $x:U\to \mathbb{R}^{n}$ and $y:V\to \mathbb{R}^{n}$ 
	defined on neighborhoods $U$ and $V$ of $p$ and $\psi(p)$, respectively, 
	such that $\tfrac{\partial}{\partial x_{i}}=X_{i}$ on $U$ and $\tfrac{\partial}{\partial y_{i}}=X_{i}$ on 
	$V$ for all $i\in \{1,...,n\}$. By Lemma \ref{nfekndkwenknfk}, there is $R=(r_{ij})\in \textup{GL}(n,\mathbb{Z})$ such that 
	$\psi_{*_{p}}X_{k}(p)=\sum_{i=1}^{n}r_{ik}X_{i}(\psi(p))$ for all $k\in \{1,...,n\}$. Thus 
	the Jacobian matrix of $\psi$ in the coordinates $x$ and $y$ is the constant matrix $R$. 
	This shows that $\psi$ is affine with respect to $\nabla$.  
\end{proof}

	We now focus our attention on the group $\textup{Diff}(M,h,L)$.  
	Given $A\in \textup{GL}(n,\mathbb{Z})$, we will denote by $\rho_{A}:\mathbb{T}^{n}\to \mathbb{T}^{n}$ 
	the Lie group isomorphism defined by $\rho_{A}([t])=[At]$, where $[t]$ denotes the equivalence class of 
	$t\in \mathbb{R}^{n}$ in $\mathbb{T}^{n}=\mathbb{R}^{n}/\mathbb{Z}^{n}$. Let $\textup{Aut}(\mathbb{T}^{n})$ 
	denote the group of Lie group isomorphisms of the torus $\mathbb{T}^{n}$. It is well-known that 
	the map $\textup{GL}(n,\mathbb{Z})\to\textup{Aut}(\mathbb{T}^{n})$, $A\mapsto \rho_{A}$ is a group isomorphism 
	(see, e.g., \cite{tammo}, Chapter IV).

	By Lemma \ref{nfekndkwenknfk}, for every $\psi\in \textup{Diff}(M,h,L)$, there is a (necessarily unique) matrix 
	$R(\psi)=(r(\psi)_{ij})_{1\leq i,j\leq n}\in \textup{GL}(n,\mathbb{Z})$ such that 
	$\psi_{*_{p}}X_{k}(p)=\sum_{i=1}^{n}r(\psi)_{ik}X_{i}(\psi(p))$ for all $p\in M$ and all $k\in \{1,...,n\}$. Clearly, the map 
	\begin{eqnarray}\label{neknkreknkfenk}
		\begin{tabular}{lrlll}
			$$ & $\textup{Diff}(M,h,L)$ & $\to$    &$\textup{Aut}(\mathbb{T}^{n}),$ \\[0.5em]
			    &   $\psi$              & $\mapsto$ & $\rho_{R(\psi)}$,
		\end{tabular}
	\end{eqnarray}
	is a group homomorphism and hence we can form the semidirect product $\mathbb{T}^{n}\rtimes \textup{Diff}(M,h,L)$.
	By definition, it is the Cartesian product $\mathbb{T}^{n}\times \textup{Diff}(M,h,L)$ together with the group multiplication
	\begin{eqnarray*}
		(a,\psi)\cdot (a',\psi')=\big(\rho_{R(\psi)}(a') a, \psi\circ \psi'\big). 
	\end{eqnarray*}
	
	Let $\Gamma$ be the group action of $\mathbb{T}^{n}\rtimes \textup{Diff}(M,h,L)$ on $M_{L}$ defined by 
	\begin{eqnarray}\label{nkwdnknefknk}
		\Gamma\Big(([t],\psi), q_{L}(u)\Big)=\Phi_{[t]}\big(q_{L}(\psi_{*_{p}}u)\big)
		=q_{L}\bigg(\psi_{*_{p}}u+\sum_{i=1}^{n}t_{i}X_{i}(\psi(p))\bigg), 
	\end{eqnarray}
	where $t=(t_{1},...,t_{n})\in \mathbb{R}^{n}$, $\psi\in \textup{Diff}(M,h,L)$ and $p=\pi(u)\in M$. 

	On $M_{L}$, there is a unique K\"{a}hler structure that turns $q_{L}:TM\to M_{L}$ into a 
	K\"{a}hler covering map (here $TM$ is endowed with the K\"{a}hler structure 
	coming from Dombrowski's construction). Let $g$ be the corresponding K\"{a}hler metric on $M_{L}$ 
	and let $\textup{Aut}(M_{L},g)^{\mathbb{T}^{n}}$ be the group of holomorphic isometries of $M_{L}$ that are 
	equivariant in the following sense: 
	for each $\varphi\in \textup{Aut}(M_{L},g)^{\mathbb{T}^{n}}$, there is a Lie group 
	isomorphism $\rho:\mathbb{T}^{n}\to \mathbb{T}^{n}$ such that 
	$\varphi\circ \Phi_{a}=\Phi_{\rho(a)}\circ \varphi$ for all $a\in \mathbb{T}^{n}$. 

	The next proposition is the main result of this section.

\begin{proposition}\label{nfekwwndkefnk}
	The map
	\begin{center}
		\begin{tabular}{lrlll}
			$$ &   $\mathbb{T}^{n}\rtimes \textup{Diff}(M,h,L)$ & 
			   $\to$    &$\textup{Aut}(M_{L},g)^{\mathbb{T}^{n}},$ \\[0.5em]
			    &   $a$              & $\mapsto$ & 
			    $\Gamma_{a}$,
		\end{tabular}
	\end{center}
	is a group isomorphism. 
\end{proposition}

	The rest of this section is devoted to the proof of Proposition \ref{nfekwwndkefnk}. For our purposes, 
	it is convenient to trivialize $TM$ via the map 
	\begin{center}
		\begin{tabular}{lrlll}
			$$ &   $f\,\,\,:\,\,\,\mathbb{R}^{n}\times M$& 
				$\to$    &$TM,$ \\[0.5em]
			    &   $ \big((u_{1},...,u_{n}),p\big)$              & $\mapsto$ & 
			    $u_{1}X_{1}(p)+...+u_{n}X_{n}(p)$.
		\end{tabular}
	\end{center}
	Let $K:TTM\to TM$ be the connector associated to $\nabla$. Given $u\in T_{p}M$, we will 
	identify $T_{u}(TM)$ and $T_{p}M\oplus T_{p}M$ via the map $A\mapsto (\pi_{*}A,KA)$ (see Section \ref{nkwwnkwnknknfkn}). 
\begin{lemma}\label{jfkjdwkjdkjk}
	Under the identification $T_{f(u,p)}(TM)=T_{p}M\oplus T_{p}M$, 
	the derivative of $f$ at $(u,p)\in \mathbb{R}^{n}\times M$ is given by 
	\begin{eqnarray*}
		f_{*_{(u,p)}}(v,w)= \big(w, v_{1}X_{1}(p)+...+v_{n}X_{n}(p)\big),
	\end{eqnarray*}
	where $v=(v_{1},...,v_{n})\in \mathbb{R}^{n}\cong T_{u}\mathbb{R}^{n}$ and $w\in T_{p}M$. 
\end{lemma}
\begin{proof}
	We must show that 
	\begin{description}
		\item[$(1)$] $\pi_{*}f_{*_{(u,p)}}(v,w)=w $ and 
		\item[$(2)$] $Kf_{*_{(u,p)}}(v,w)=v_{1}X_{1}(p)+...+v_{n}X_{n}(p)$. 
	\end{description}
	Let $p(t)$ be a smooth curve in $M$ such that $p(0)=p$ and $\tfrac{dp}{dt}(0)=w$. We have:
	\begin{eqnarray*}
		\pi_{*}f_{*_{(u,p)}}(v,w)=\dfrac{d}{dt}\bigg\vert_{0}(\pi\circ f)(u+tv,p(t))=\dfrac{d}{dt}\bigg\vert_{0}p(t)=w. 
	\end{eqnarray*}
	This shows $(1)$. For $(2)$, let $(x_{1},...,x_{n})$ be an affine coordinate system with respect 
	to $\nabla$ defined in a connected neighborhood $U\subseteq M$ of $p$. 
	Set $Z(t)=f(u+tv,p(t))$. Because each $X_{k}$ is parallel, there are real 
	numbers $a_{ij}$, $1\leq i,j\leq n$, such that 
	\begin{eqnarray}\label{nfkdnwkenfk}
		X_{k}=\sum_{j=1}^{n}a_{kj}\dfrac{\partial}{\partial x_{j}}
	\end{eqnarray}
	on $U$ and hence 
	\begin{eqnarray*}
		Z(t)=\sum_{k,j=1}^{n}(u_{k}+tv_{k})a_{kj}\dfrac{\partial}{\partial x_{j}}\bigg\vert_{p(t)}.
	\end{eqnarray*}
	Let $(q,r)=(q_{1},...,q_{n},r_{1},...,r_{n})$ be the local coordinates on $TM$ canonically associated to 
	$(x_{1},...,x_{n})$ (see Section \ref{nkwwnkwnknknfkn}). In the coordinates $(q,r)$, $Z(t)$ reads
	\begin{eqnarray*}
		Z(t)=\Big(p_{1}(t),...,p_{n}(t), \sum_{k=1}^{n}(u_{k}+tv_{k})a_{k1},...,\sum_{k=1}^{n}(u_{k}+tv_{k})a_{kn}\Big),
	\end{eqnarray*}
	where $p_{k}(t)=x_{k}(p(t))$. Writing $w=\sum_{i=1}^{n}w_{i}\tfrac{\partial}{\partial x_{i}}\big\vert_{p}$, 
	it follows that 
	\begin{eqnarray*}
		\dfrac{d}{dt}\bigg\vert_{0}Z(t)= \Big(w_{1},...,w_{n}, \sum_{k=1}^{n}v_{k}a_{k1},...,\sum_{k=1}^{n}v_{k}a_{kn}\Big), 
	\end{eqnarray*}
	or equivalently, that 
	\begin{eqnarray*}
		\dfrac{d}{dt}\bigg\vert_{0}Z(t)= \sum_{j=1}^{n}w_{j}\dfrac{\partial}{\partial q_{j}}\bigg\vert_{f(u,p)}+
		\sum_{j=1}^{n}\sum_{k=1}^{n}v_{k}a_{kj}\dfrac{\partial }{\partial r_{j}}\bigg\vert_{f(u,p)}.
	\end{eqnarray*}
	Since $K\tfrac{\partial}{\partial q_{j}}=0$ and $K\tfrac{\partial}{\partial r_{j}}=\tfrac{\partial}{\partial x_{j}}$ (see 
	Section \ref{nkwwnkwnknknfkn}), it follows from the linearity of $K:T_{f(u,p)}TM\to T_{p}M$ that 
	\begin{eqnarray*}
		Kf_{*_{(u,p)}}(v,w)=  K\dfrac{d}{dt}\bigg\vert_{0}Z(t)= 
		\sum_{j=1}^{n}\sum_{k=1}^{n}v_{k}a_{kj}\dfrac{\partial }{\partial x_{j}}\bigg\vert_{p} = \sum_{k=1}^{n}v_{k}X_{k}(p),
	\end{eqnarray*}
	where we have used \eqref{nfkdnwkenfk}. This shows (2) and concludes the proof of the lemma.
\end{proof}

	Let $(\overline{g},\overline{J}, \overline{\omega})$ be the unique K\"{a}hler structure on $\mathbb{R}^{n}\times M$ 
	that makes $f:\mathbb{R}^{n}\times M\to TM$ a K\"{a}hler isomorphism 
	(here $TM$ is endowed with the K\"{a}hler structure associated to $(h,\nabla)$ via 
	Dombrowski's construction).

\begin{lemma}\label{cnknksnkcnsknk}
	Let $p\in M$, $u,v,w\in \mathbb{R}^{n}$ and $A,B$ in $T_{p}M$. 
	Write $v=(v_{1},..,v_{n})$ and $w=(w_{1},...,w_{n})$. The following holds. 
	\begin{enumerate}[(1)]
	\item $\overline{g}_{(u,p)}\big((v,A),(w,B)\big)=\sum_{i,j=1}^{n}v_{i}w_{j}h_{p}(X_{i},X_{j})
			+h_{p}(A,B).$
	\item $\overline{J}_{(u,p)}\big(v, \sum_{k=1}^{n}w_{k}X_{k}(p)\big)=
			\big(w, -(v_{1}X_{1}(p)+...+v_{n}X_{n}(p))\big). $
	\end{enumerate}
\end{lemma}
\begin{proof}
	By inspection of Dombrowski's construction together with Lemma \ref{jfkjdwkjdkjk}. 
\end{proof}

	Let $q$ denote the quotient map $\mathbb{R}^{n}\to \mathbb{T}^{n}=\mathbb{R}^{n}/\mathbb{Z}^{n}$. Consider the diagram 
\begin{eqnarray*}
	\begin{tikzcd}
		\mathbb{R}^{n}\times M \arrow{r}{\displaystyle f} \arrow[swap]{d}{ q\times Id } & TM \arrow{d}{\displaystyle q_{L}} \\
		\mathbb{T}^{n}\times M \arrow[swap]{r}{\displaystyle \tilde{f} } & M_{L}
	\end{tikzcd}
\end{eqnarray*}
	where $(q\times Id)(u,p)=(q(u),p)$ and $\tilde{f}$ is the unique diffeomorphism that makes the diagram commutative. 
	The following result is immediate.	

\begin{lemma}\label{nkedwknknkdnk}
	Suppose $\mathbb{R}^{n}\times M$ (resp. $\mathbb{T}^{n}\times M$) 
	is endowed with the unique K\"{a}hler structure that makes $f$ (resp. $\tilde{f}$) a K\"{a}hler isomorphism. 
	Then, 
	\begin{enumerate}[(1)]
		\item $q\times Id$ is a holomorphic and locally isometric covering map.  
		\item $\tilde{f}$ is equivariant: for every $a\in \mathbb{T}^{n}$, 
			$\tilde{f}\circ \Psi_{a}=\Phi_{a}\circ \tilde{f},$ where $\Psi$ 
			is the left action of $\mathbb{T}^{n}$ on $\mathbb{T}^{n}\times M$ defined by $\Psi(a,(b,p))=(ab,p)$. 
	\end{enumerate}
\end{lemma}
%

	Let $g'=\tilde{f}^{*}g$ denote the K\"{a}hler metric on $\mathbb{T}^{n}\times M$, 
	and let $\textup{Aut}(\mathbb{T}^{n}\times M,g')^{\mathbb{T}^{n}}$ be the group of holomorphic and isometric diffeomorphisms 
	$F$ of $\mathbb{T}^{n}\times M$ that are equivariant in the following sense: for each $F\in \textup{Aut}(\mathbb{T}^{n}\times M)^{\mathbb{T}^{n}}$, 
	there is a Lie group isomorphism $\rho:\mathbb{T}^{n}\to \mathbb{T}^{n}$ such that 
	\begin{eqnarray*}	
		F\circ \Psi_{a}=\Psi_{\rho(a)}\circ F
	\end{eqnarray*}
	for all $a\in \mathbb{T}^{n}$, where $\Psi$ is the action described in Lemma \ref{nkedwknknkdnk}. 
	Because $\tilde{f}$ is an equivariant K\"{a}hler isomorphism, 
	the map 
	\begin{eqnarray}\label{vneknwknek}
		\textup{Aut}(\mathbb{T}^{n}\times M,g')^{\mathbb{T}^{n}}\to \textup{Aut}(M_{L},g)^{\mathbb{T}^{n}},\,\,\,\,
		F\mapsto \tilde{f}\circ F\circ (\tilde{f})^{-1}
	\end{eqnarray}
	is a group isomorphism, and so we can work with either $\textup{Aut}(\mathbb{T}^{n}\times M,g')^{\mathbb{T}^{n}}$ or $\textup{Aut}(M_{L},g)^{\mathbb{T}^{n}}$, 
	whichever is more convenient. 
	
	Let $F\in \textup{Aut}(\mathbb{T}^{n}\times M,g')^{\mathbb{T}^{n}}$ be fixed. Since $F$ is equivariant, there is a smooth 
	function $\phi:M\to \mathbb{T}^{n}$, a smooth diffeomorphism 
	$\psi:M\to M$ and a Lie group isomorphism $\rho:\mathbb{T}^{n}\to \mathbb{T}^{n}$ such that 
	\begin{eqnarray*}
		F(a,p)=\big(\rho(a)\phi(p), \psi(p)\big)
	\end{eqnarray*}
	for all $a\in \mathbb{T}^{n}$ and all $p\in M$. Moreover, the fact that $\rho$ is a Lie group isomorphism of the torus implies that 
	there is $A\in \textup{GL}(n,\mathbb{Z})$ such that $\rho=\rho_{A}$.

\begin{lemma}\label{nkwdnkenknk}
	The equivariant diffeomorphism $F(a,p)=\big(\rho(a)\phi(p), \psi(p)\big)$
	is holomorphic if and only if the following conditions are satisfied:
	\begin{enumerate}[(1)]
	\item $\phi:M\to \mathbb{T}^{n}$ is constant,
	\item $\psi$ preserves $L$ and $R(\psi)=A$. 
	\end{enumerate}
\end{lemma}
\begin{proof}
	Let $J$ be the complex structure on $\mathbb{T}^{n}\times M$. Let $(q(u),p)\in \mathbb{T}^{n}\times M$ be arbitrary, where 
	$u\in \mathbb{R}^{n}$. Since $q:\mathbb{R}^{n}\to \mathbb{T}^{n}$ is a covering map, there exist an open neighborhood 
	$U$ of $p$ in $M$ and a smooth map $\tilde{\phi}:U\to \mathbb{R}^{n}$ such that $\phi=q\circ \tilde{\phi}$ on $U$. 
	From this and the fact that $q$ is a Lie group homomorphism 
	it is easy to compute the derivative of $F$ at $(q(u),p)$ in the direction $(q_{*_{u}}v,Z)$: 	
	\begin{eqnarray*}
		F_{*_{(q(u),p)}}(q_{*_{u}}v,Z)&=&\dfrac{d}{dt}\bigg\vert_{0} F(q(u+tv),p(t))\\
		&=&\dfrac{d}{dt}\bigg\vert_{0} \big(\rho(q(u+tv))q(\tilde{\phi}(p(t))),\psi(p(t))\big)\\
		&=&\dfrac{d}{dt}\bigg\vert_{0} \big( q(Au+tAv+\tilde{\phi}(p(t))), \psi(p(t)) \big) \,\,\,\,\,\,\,\,\,\,\,\,\,\,(\rho=\rho_{A})\\
		&=& (q\times Id)_{*_{(Au+\tilde{\phi}(p), \psi(p))}}\big(Av+\tilde{\phi}_{*_{p}}Z,\psi_{*_{p}}Z\big),
	\end{eqnarray*}
	where $p(t)$ is a smooth curve in $M$ such that $p(0)=p$ and $\tfrac{dp(t)}{dt}(0)=Z$. 
	Thus 
        \begin{eqnarray}\label{nwnwknekfnkwfnkn}
		F_{*_{(q(u),p)}}(q_{*_{u}}v,Z)=(q\times Id)_{*}\big(Av+\tilde{\phi}_{*_{p}}Z,\psi_{*_{p}}Z\big). 
	\end{eqnarray}
	It follows from this, Lemma \ref{cnknksnkcnsknk} and the fact that $q\times Id$ is holomorphic that 
	\begin{eqnarray}
		\lefteqn{J F_{*_{(q(u),p)}}(q_{*_{u}}v,Z) = J(q\times Id)_{*}\big(Av+\tilde{\phi}_{*_{p}}Z,\psi_{*_{p}}Z\big)}\nonumber\\
		&=& (q\times Id)_{*}\overline{J}(Av+\tilde{\phi}_{*_{p}}Z,\psi_{*_{p}}Z) \nonumber \\
		&=&(q\times Id)_{*} \bigg(\big((\psi_{*_{p}}Z)_{1},...,(\psi_{*_{p}}Z)_{n}\big), 
			-\sum_{k=1}^{n}(Av+\tilde{\phi}_{*_{p}}Z)_{k}X_{k}(\psi(p))\bigg),\label{nfekndkwnk}
	\end{eqnarray}
	where $(\psi_{*_{p}}Z)_{1},...,(\psi_{*_{p}}Z)_{n}$ are the coordinates of $\psi_{*_{p}}Z$ with respect to the basis 
	$X_{1}(\psi(p)),...,X_{k}(\psi(p))$. On the other hand, 
	\begin{eqnarray}
		F_{*_{(q(u),p)}}J(q_{*_{u}}v,Z)&=&F_{*_{(q(u),p)}}J(q\times Id)_{*}(v,Z)\nonumber \\
		&=&F_{*_{(q(u),p)}}(q\times Id)_{*}\overline{J}(v,Z)\nonumber \\
		&=&F_{*_{(q(u),p)}}(q\times Id)_{*} \big((Z_{1},...,Z_{n}),-V\big) \nonumber \\
		&=& F_{*_{(q(u),p)}}\big(q_{*_{u}}(Z_{1},...,Z_{n}),-V\big) \nonumber \\
		&=& (q\times Id)_{*} \big(A(Z_{1},...,Z_{n})-\tilde{\phi}_{*_{p}}V, -\psi_{*_{p}}V\big), \,\,\,\,\,\,\,\,\,\,(\textup{see}\,\, 
			\eqref{nwnwknekfnkwfnkn})\label{nfekwdnkednk}
	\end{eqnarray}
	where $V=\sum_{k}v_{k}X_{k}(p)$ and $Z_{1},...,Z_{n}$ are the coordinates of $Z$ with respect to the basis $X_{1}(p),...,
	X_{n}(p)$. Comparing \eqref{nfekndkwnk} with \eqref{nfekwdnkednk} we see that $F_{*}$ and $J$ commute at $(q(u),p)$ if and only if
	\begin{eqnarray*}
		(S)\,\,\,\,
	\left \lbrace
		\begin{array}{lll}
			\big((\psi_{*_{p}}Z)_{1},...,(\psi_{*_{p}}Z)_{n}\big)       & = &A(Z_{1},...,Z_{n})-\tilde{\phi}_{*_{p}}V, \\[0.5em]
			\displaystyle\sum_{k=1}^{n}(Av+\tilde{\phi}_{*_{p}}Z)_{k}X_{k}(\psi(p)) & = &\psi_{*_{p}}V
		\end{array}
	\right.
	\end{eqnarray*}
	for all $Z\in T_{p}M$ and all $v=(v_{1},...,v_{n})\in \mathbb{R}^{n}$. By inspection of $(S)$ we deduce that $F$ is holomorphic 
	at $(q(u),p)$ if and only if $\tilde{\phi}_{*_{p}}=0$ ($\Leftrightarrow\,\phi_{*_{p}}=0$) and 
	$\psi_{*_{p}}X_{k}(p)=\sum_{i=1}^{n}A_{ik}X_{i}(\psi(p))$ for all $k=1,...,n$. Since $M$ is connected, the condition $\phi_{*_{p}}=0$ 
	for all $p\in M$ is equivalent to $\phi$ being constant. The other condition means that $\psi$ preserves $L$ and that 
	$R(\psi)=A$. The lemma follows.  
\end{proof}

\begin{lemma}\label{nfeknwdkenkfn}
	Let $F(a,p)=\big(\rho(a)\phi(p), \psi(p)\big)$ be as in the preceding lemma. 
	Suppose $F$ holomorphic. Then $F$ is isometric if and only if $\psi^{*}h=h$. 
\end{lemma}
\begin{proof}
	Recall that $\overline{g}$ (resp. $g'$) denotes the Riemannian metric on $\mathbb{R}^{n}\times M$ 
	(resp. $\mathbb{T}^{n}\times M$). We must show that $F^{*}g'=g'$ if and only if $\psi^{*}h=h$. 
	
	In the proof of Lemma \ref{nkwdnkenknk}, we have computed the derivative of $F$ at 
	$(q(u),p)\in \mathbb{T}^{n}\times M$ in the direction $(q_{*_{u}}v,Z)\in T_{q(u)}\mathbb{T}^{n}\times T_{p}M$ 
	(see \eqref{nwnwknekfnkwfnkn}). Taking into account the fact that $\phi$ is constant by Lemma \ref{nkwdnkenknk}, 
	this formula becomes 
	\begin{eqnarray*}
		F_{*_{(q(u),p)}}(q_{*_{u}}v,Z)=(q\times Id)_{*_{(Au+\tilde{\phi}(p), \psi(p))}}
		\big(Av,\psi_{*_{p}}Z\big),
	\end{eqnarray*}
	where 	
	$\tilde{\phi}$ is a smooth function defined on some neighborhood of $p$ in $M$ such that $\phi=q\circ \tilde{\phi}$. 
	
	We now compute $F^{*}g'$ at $(q(u),p)\in \mathbb{T}^{n}\times M$. Given two pairs 
	$(q_{*_{u}}v,V)$ and $(q_{*_{u}}w,W)$ in $T_{q(u)}\mathbb{T}^{n}\times T_{p}M$, we have 

	\begin{eqnarray*}
		\lefteqn{(F^{*}g')_{(q(u),p)}\big((q_{*_{u}}v,V), (q_{*_{u}}w,W)\big)}\\
		&=&g'_{F(q(u),p)}\big(F_{*}(q_{*_{u}}v,V), F_{*}(q_{*_{u}}w,W)\big)\\
		&=&g'_{F(q(u),p)}\Big( (q\times Id)_{*}(Av, \psi_{*}V), 
			(q\times Id)_{*}(Aw, \psi_{*}W)\Big)\\
		&=& \big((q\times Id)^{*}g'\big)\big((Av, \psi_{*}V), 
			(Aw, \psi_{*}W)\big)\\
		&=& \overline{g}\big((Av, \psi_{*}V), (Aw, \psi_{*}W)\big)\\
		&=& \sum_{i,j=1}^{n} (Av)_{i} (Aw)_{j}h_{\psi(p)}(X_{i}, X_{j}) + h_{\psi(p)}(\psi_{*}V,\psi_{*}W)\\
		&=& \sum_{i,j=1}^{n}\sum_{k,l=1}^{n} A_{ik}v_{k}A_{jl}w_{l} h_{\psi(p)}(X_{i}, X_{j}) + (\psi^{*}h)_{p}(Z,W), 
	\end{eqnarray*}
	where we have used Lemma \ref{cnknksnkcnsknk} and the fact that $q\times Id$ is a local isometry by Lemma \ref{nkedwknknkdnk}. 
	Because $F$ is holomorphic, $\psi_{*_{p}}X_{k}(p)=\sum_{i=1}^{n}A_{ik}X_{i}(\psi(p))$ for every $k=1,...,n$ and hence the 
	double sum above can be rewritten as 
	\begin{eqnarray*}
		\lefteqn{\sum_{i,j=1}^{n}\sum_{k,l=1}^{n} A_{ik}v_{k}A_{jl}w_{l} h_{\psi(p)}(X_{i}, X_{j})}\\
		&=&\sum_{k,l=1}^{n} v_{k}w_{l}h_{\psi(p)}\big(\psi_{*_{p}}X_{i}, \psi_{*_{p}}X_{j}\big)  
			= \sum_{k,l=1}^{n} v_{k}w_{l} (\psi^{*}h)_{p}(X_{i},X_{j}). 
	\end{eqnarray*}
	It follows that 
	\begin{eqnarray*}
		\lefteqn{(F^{*}g')_{(q(u),p)}\big((q_{*_{u}}v,V), (q_{*_{u}}w,W)\big)}\\
		&=&\sum_{k,l=1}^{n} v_{k}w_{l} (\psi^{*}h)_{p}(X_{i},X_{j}) + (\psi^{*}h)_{p}(V,W).
	\end{eqnarray*}
	Comparing with Lemma \ref{cnknksnkcnsknk} and using the fact that $(q\times Id)^{*}g'=\overline{g}$, 
	we see that $(F^{*}g')_{(q(u),p)}= g'_{(q(u),p)}$ if and only if 
	\begin{eqnarray*}
		\sum_{k,l=1}^{n} v_{k}w_{l} (\psi^{*}h)_{p}(X_{i},X_{j}) + (\psi^{*}h)_{p}(V,W)=
		\sum_{k,l=1}^{n} v_{k}w_{l} h_{p}(X_{i},X_{j}) + h_{p}(V,W)
	\end{eqnarray*}
	for all $v,w\in \mathbb{R}^{n}$ and all $V,W\in T_{p}M$, which is equivalent to $(\psi^{*}h)_{p}=h_{p}$. The lemma follows.  
\end{proof}

%

\begin{proof}[Proof of Proposition \ref{nfekwwndkefnk}]
	Let $\widetilde{\Gamma}$ denote the group action of $\mathbb{T}^{n}\rtimes \textup{Diff}(M,h,L)$ on $\mathbb{T}^{n}\times M$ 
	defined by 
	\begin{eqnarray}
		\widetilde{\Gamma}\big((a,\psi),(b,p)\big)=\big(\rho_{R(\psi)}(b)a, \psi(p)\big).
	\end{eqnarray}
	A simple computation shows that each $\widetilde{\Gamma}_{(a,\psi)}$ is equivariant in the sense that 
	$\widetilde{\Gamma}_{(a,\psi)}\circ \Psi_{b}= \Psi_{\rho_{R(\psi)}(b)}\circ \widetilde{\Gamma}_{(a,\psi)}$ for all 
	$b\in \mathbb{T}^{n}$, and by the two lemmas above, $\widetilde{\Gamma}_{(a,\psi)}$ is also holomorphic and isometric. 
	Therefore the map 
	\begin{center}
		\begin{tabular}{lrlll}
			$\Omega\,:$ &   $\mathbb{T}^{n}\rtimes \textup{Diff}(M,h,L)$ & 
			   $\to$    &$\textup{Aut}(\mathbb{T}^{n}\times M,g')^{\mathbb{T}^{n}},$ \\[0.5em]
			    &   $g$              & $\mapsto$ & 
			    $\widetilde{\Gamma}_{g}$,
		\end{tabular}
	\end{center}
	is a well-defined group homomorphism. Again by the two lemmas above, $\Omega$ is surjective, and it is straightforward to 
	check that $\Omega$ is injective. Therefore $\Omega$ is a group isomorphism. The rest of the proof consists in 
	using the group isomorphism $\textup{Aut}(\mathbb{T}^{n}\times M,g')^{\mathbb{T}^{n}}\to \textup{Aut}(M_{L},g)^{\mathbb{T}^{n}}$ 
	defined in \eqref{vneknwknek} and to check that 
	\begin{eqnarray*}
		\tilde{f}\circ \widetilde{\Gamma}_{(a,\psi)}\circ (\tilde{f})^{-1}=\Gamma_{(a,\psi)}
	\end{eqnarray*}
	for all $(a,\psi)\in \mathbb{T}^{n}\rtimes \textup{Diff}(M,h,L)$. The details are left to the reader. 
	This concludes the proof of Proposition \ref{nfekwwndkefnk}. 
\end{proof}

\section{Equivariant holomorphic isometries of a torification}\label{nekneknfknknf}

	

	Let $(M,h,\nabla)$ be a connected dually flat manifold
	and $\Phi:\mathbb{T}^{n}\times N\to N$ a regular torification, with K\"{a}hler metric $g$ and fundamental 
	lattice $\mathscr{L}\subset TM$ (see Section \ref{nfeknkfnekndk}). 
	Given a compatible covering map $\tau:TM\to N^{\circ}$ and $\psi\in \textup{Diff}(M,h,\nabla)$,
	we will denote by 
	\begin{itemize}
	\item $\textup{lift}_{\tau}(\psi):N\to N$ 
			the unique lift of $\psi$ that satisfies $\textup{lift}_{\tau}(\psi)\circ \tau =\tau\circ \psi_{*}$ 
			on $TM$,
		\item $\rho_{\tau}(\psi):\mathbb{T}^{n}\to \mathbb{T}^{n}$ the unique Lie group homomorphism satisfying 
			$\textup{lift}_{\tau}(\psi)\circ \Phi_{a}=\Phi_{\rho_{\tau}(\psi)(a)}\circ \textup{lift}_{\tau}(\psi)$ 
			for all $a\in \mathbb{T}^{n}$. 
	\end{itemize}
	Let $\textup{Aut}(N,g)^{\mathbb{T}^{n}}$ be the group of holomorphic isometries of $(N,g)$
	that are equivariant in the following sense: for each $\varphi\in \textup{Aut}(N,g)^{\mathbb{T}^{n}}$, 
	there is a Lie group isomorphism $\rho:\mathbb{T}^{n}\to \mathbb{T}^{n}$ 
	such that $\varphi\circ \Phi_{a}=\Phi_{\rho(a)}\circ \varphi$ for all $a\in \mathbb{T}^{n}$. 
	In \cite{molitor-toric}, it is shown that the maps
	\begin{center}
		\begin{tabular}{lclllllll}
			$\textup{lift}_{\tau}$ &: &   $\textup{Diff}(M,h,\nabla)$ & 
			   $\to$    &$\textup{Aut}(N,g)^{\mathbb{T}^{n}},$ & $\psi$ & $\mapsto$ & $\textup{lift}_{\tau}(\psi),$ \\[0.5em]
			$\rho_{\tau}$ &: &   $\textup{Diff}(M,h,\nabla)$ & $\to$    &$\textup{Aut}(\mathbb{T}^{n}),$   
			    &   $\psi$ & $\mapsto$ & $\rho_{\tau}(\psi),$ 
		\end{tabular}
	\end{center}
	are group homomorphisms. 

	Below we use the notation $\rho_{A}$ to denote the Lie group isomorphism of the 
	torus $\mathbb{T}^{n}=\mathbb{R}^{n}/\mathbb{Z}^{n}$ defined by $\rho_{A}([t])=[At]$, 
	where $A\in \textup{GL}(n,\mathbb{Z})$ and $t\in \mathbb{R}^{n}$.

\begin{lemma}\label{nkwnknefknknkhhh}
	Let $(\mathscr{L},X,F)$ be a toric parametrization, with corresponding 
	toric factorization $\pi=\kappa\circ \tau$. Let $\psi\in \textup{Diff}(M,h,\nabla)$ be arbitrary. 
	If $\rho_{\tau}(\psi)=\rho_{A}$, where $A\in \textup{GL}(n,\mathbb{Z})$, then 
	\begin{eqnarray}\label{nnkefnknkfnkn}
		\psi_{*}X_{i}(p)=\sum_{j=1}^{n}A_{ji}X_{j}(\psi(p))
	\end{eqnarray}
	for all $p\in M$ and $i\in \{1,...,n\}$. In particular, $\psi$ preserves $\mathscr{L}$. 
\end{lemma}
\begin{proof}
	For simplicity, write $\rho=\rho_{A}$. 
	Upon the identification $\textup{Lie}(\mathbb{T}^{n})=\mathbb{R}^{n}$ given by the derivative at $0$ of the quotient 
	map $\mathbb{R}^{n}\to \mathbb{R}^{n}/\mathbb{Z}^{n}=\mathbb{T}^{n}$, 
	the derivative of $\rho$ at $e$ in the direction 
	$t\in \mathbb{R}^{n}$ is given by $\rho_{*_{e}}t=At$. It follows from this and 
	Proposition \ref{nfekwnknfkenk} that 
	\begin{eqnarray*}
		\psi_{*}\circ (T_{X})_{t}=(T_{X})_{At}\circ \psi_{*}
	\end{eqnarray*}
	for all $t\in\mathbb{R}^{n}\cong \textup{Lie}(\mathbb{T}^{n})$, where $T_{X}$ is the group action of 
	$\mathbb{R}^{n}$ on $TM$ defined by $T_{X}(t,u)=u+\sum_{k=1}^{n}t_{k}X_{k}$. Taking $t=(0,..,1,...,0)$ 
	and evaluating at $0$ yields \eqref{nnkefnknkfnkn}. 
\end{proof}
	
	An immediate consequence of Lemmas \ref{nnknefkwnknefknk} and \ref{nkwnknefknknkhhh} is the following 

\begin{proposition}\label{neknknefknkn}
	$\textup{Diff}(M,h,\nabla)=\textup{Diff}(M,h,\mathscr{L})$. 
\end{proposition}

	Under the hypotheses of Lemma \ref{nkwnknefknknkhhh}, there are two group homomorphisms 
	$\textup{Diff}(M,h,\nabla)\to \textup{Aut}(\mathbb{T}^{n})$ that one may naturally consider, namely 
	\eqref{neknkreknkfenk} and $\rho_{\tau}$. By Lemma \ref{nkwnknefknknkhhh}, 
	these homomorphisms are equal, and thus one can use either 
	of them to define the semidirect product $\mathbb{T}^{n}\rtimes \textup{Diff}(M,h,\nabla)$. 
	In terms of $\rho_{\tau}$, group multiplication reads 
	\begin{eqnarray}\label{neknkngekfnwknfk}
		(a,\psi)\cdot (a',\psi')=\big(\rho_{\tau}(\psi)(a') a, \psi\circ \psi'\big). 
	\end{eqnarray}

\begin{theorem}\label{nfkwndkefnknkwn}
	Let $\Phi:\mathbb{T}^{n}\times N\to N$ be a regular torification of a connected dually flat manifold $(M,h,\nabla)$, 
	with K\"{a}hler metric $g$. Given a compatible covering map $\tau:TM\to N^{\circ}$, the map
	\begin{center}
		\begin{tabular}{lrlll}
			$$ &   $\mathbb{T}^{n}\rtimes \textup{Diff}(M,h,\nabla)$ & 
			   $\to$    &$\textup{Aut}(N,g)^{\mathbb{T}^{n}},$ \\[0.5em]
			    &   $(a,\psi)$              & $\mapsto$ & 
			    $\Phi_{a}\circ \textup{lift}_{\tau}(\psi)$,
		\end{tabular}
	\end{center}
	is a group isomorphism, where the group structure of the semidirect product is given by \eqref{neknkngekfnwknfk}.
\end{theorem}

	Before we proceed with the proof, let us establish some notation. K\"{a}hler metrics on $N,$ $N^{\circ}$ and $M_{\mathscr{L}}$ are denoted by the same symbol $g$. 
	Given a K\"{a}hler manifold $(W,g)$ equipped with a torus action $\Phi:\mathbb{T}^{n}\times W\to W$, 
	we will denote by $\textup{Aut}(W,g)^{\mathbb{T}^{n}}$ the set of holomorphic isometries of $W$ that are equivariant in the following sense: 
	for every $\varphi\in \textup{Aut}(W,g)^{\mathbb{T}^{n}}$, there is a Lie group isomorphism $\rho:\mathbb{T}^{n}\to \mathbb{T}^{n}$ such that 
	$\varphi\circ \Phi_{a}=\Phi_{\rho(a)}\circ \varphi$ for all $a\in \mathbb{T}^{n}$.

\begin{lemma}\label{neknknfknknkn}
	The map $\textup{Aut}(N,g)^{\mathbb{T}^{n}}\to \textup{Aut}(N^{\circ},g)^{\mathbb{T}^{n}}$, 
	$\varphi\mapsto \varphi\vert_{N^{\circ}}$ is a group isomorphism, where $\varphi\vert_{N^{\circ}}$ denotes the restriction of $\varphi$ to $N^{\circ}$. 
\end{lemma}
\begin{proof}
	Let $\varphi\in \textup{Aut}(N,g)^{\mathbb{T}^{n}}$ be arbitrary. First we show that $\varphi(N^{\circ})=N^{\circ}$. 
	Let $p\in N^{\circ}$ be arbitrary, and let $a\in \mathbb{T}^{n}$ be an element of the isotropy group of 
	$\varphi(p)$, that is, $\Phi(a,\varphi(p))=\varphi(p)$. Since $\varphi$ is equivariant by hypothesis, 
	there is a Lie group isomorphism 
	$\rho:\mathbb{T}^{n}\to \mathbb{T}^{n}$ such that $\varphi\circ \Phi_{t}=\Phi_{\rho(t)}\circ \varphi$ for all 
	$t\in \mathbb{T}^{n}$. Since $\rho$ is surjective, there is $b\in \mathbb{T}^{n}$ 
	such that $a=\rho(b)$. Thus
        \begin{eqnarray*}
		\Phi(a,\varphi(p))=\varphi(p) \,\,\,\,\,\,\,&\Rightarrow& \,\,\,\,\,\,\,(\Phi_{\rho(b)}\circ \varphi)(p)=\varphi(p) \,\,\,\,\,\,\,\,\,\,(a=\rho(b))\\
		&\Rightarrow& \,\,\,\,\,\,\,(\varphi\circ \Phi_{b})(p)=\varphi(p) \,\,\,\,\,\,\,\,\,\,(\textup{$\varphi$ is equivariant})\\
		&\Rightarrow& \,\,\,\,\,\,\,\Phi_{b}(p)=p \,\,\,\,\,\,\,\,\,\,(\textup{$\varphi$ is a diffeomorphism})\\
		&\Rightarrow& \,\,\,\,\,\,\,b=e\,\,\,\,\,\,\,\,\,\,(\textup{$\Phi$ is free at $p$})\\
		&\Rightarrow& \,\,\,\,\,\,\,a=\rho(b)=\rho(e)=e. 
	\end{eqnarray*}
	This shows that $\Phi$ is free at $\varphi(p)$. Thus $\varphi(N^{\circ})\subseteq N^{\circ}$. The same argument 
	with $\varphi^{-1}$ in place of $\varphi$ also shows that $\varphi(N^{\circ})\supseteq N^{\circ}$. Thus $\varphi(N^{\circ})=N^{\circ}$. 

	It follows that the restriction of $\varphi$ to $N^{\circ}$ is a holomorphic and isometric diffeomorphism from $N^{\circ}$ onto $N^{\circ}$. 
	Moreover, it is clear that this restriction is equivariant. Thus the map $\textup{Aut}(N,g)^{\mathbb{T}^{n}}\to \textup{Aut}(N^{\circ},g)^{\mathbb{T}^{n}}$ 
	is well defined. The latter is obviously a group homomorphism, and because $N^{\circ}$ is dense in $N$, it is injective. 
	It remains to show that it is surjective. So let $\varphi\in  \textup{Aut}(N^{\circ},g)^{\mathbb{T}^{n}}$ be arbitrary. 
	Because the K\"{a}hler metric on $N$ is real analytic, the isometry $\varphi$ extends uniquely 
	to an isometry $\widetilde{\varphi}:N\to N$ (see \cite{Kobayashi-Nomizu}, Chapter VI, Corollary 6.4). 
	Since $\widetilde{\varphi}$ is holomorphic and equivariant on $N^{\circ}$, which is dense in $N$, 
	$\widetilde{\varphi}$ is also holomorphic and equivariant on $N.$ 
	Thus $\widetilde{\varphi}\in \textup{Aut}(N,g)^{\mathbb{T}^{n}}$. 
	Since $\widetilde{\varphi}\vert_{N^{\circ}}=\varphi$, this shows that 
	the map $\textup{Aut}(N,g)^{\mathbb{T}^{n}}\to \textup{Aut}(N^{\circ},g)^{\mathbb{T}^{n}}$ is surjective. The lemma follows. 
\end{proof}

\begin{proof}[Proof of Theorem \ref{nfkwndkefnknkwn}]
	Let $(\mathscr{L},X,F)$ be a toric parametrization with corresponding toric factorization $\pi=\kappa\circ \tau.$
	Consider the following composition of group homomorphisms
        \begin{eqnarray*}
		\mathbb{T}^{n}\rtimes \textup{Diff}(M,h,\nabla)\overset{\varepsilon_{1}}{\longrightarrow} 
		\textup{Aut}(M_{\mathscr{L}},g)^{\mathbb{T}}\overset{\varepsilon_{2}}{\longrightarrow} \textup{Aut}(N^{\circ},g)^{\mathbb{T}} 
		\overset{\varepsilon_{3}}{\longrightarrow} \textup{Aut}(N,g)^{\mathbb{T}}, 
	\end{eqnarray*}
	where 
	\begin{itemize}
	\item $\varepsilon_{1}((a,\psi))(q_{\mathscr{L}}(u))=(\Phi_{X})_{a}\big(q_{\mathscr{L}}(\psi_{*}u)\big)$, 
	\item $\varepsilon_{2}(\varphi)= F\circ \varphi\circ F^{-1}$, 
	\item $\varepsilon_{3}$ is the inverse of the restriction map $\varphi\mapsto \varphi\vert_{N^{\circ}}$.
	\end{itemize}

	These three maps are group isomorphisms (this follows from Propositions \ref{nfekwwndkefnk} and \ref{neknknefknkn}, Lemma \ref{neknknfknknkn} and 
	the fact that $F$ is equivariant). Therefore $\varepsilon:=\varepsilon_{3}\circ \varepsilon_{2}\circ \varepsilon_{1}$ is a group isomorphism. 

	Let $a\in \mathbb{T}^{n}$ and $\psi\in \textup{Diff}(M,h,\nabla)$ be arbitrary. 
	Given $p=F(q_{\mathscr{L}}(u))\in N^{\circ}$, we compute: 
        \begin{eqnarray*}
		\big(\varepsilon_{2}(\varepsilon_{1}(a,\psi))\big)(p)&=& (F\circ \varepsilon_{1}(a,\psi)\circ F^{-1})(p) \,\,\,\,\,\,\,\textup{(definition of $\varepsilon_{2}$)} \\
		&=&(F\circ \varepsilon_{1}(a,\psi))(q_{\mathscr{L}}(u))\\
		&=& (F\circ (\Phi_{X})_{a})(q_{\mathscr{L}}(\psi_{*}u)) \,\,\,\,\,\,\,\textup{(definition of $\varepsilon_{1}$)}\\
		&=& (\Phi_{a}\circ F\circ q_{\mathscr{L}})(\psi_{*}u)) \,\,\,\,\,\,\,\textup{($F$ is equivariant)}\\
		&=& (\Phi_{a}\circ \tau)(\psi_{*}u) \,\,\,\,\,\,\,\textup{($\tau=F\circ q_{\mathscr{L}}$)}\\
		&=& (\Phi_{a}\circ \textup{lift}_{\tau}(\psi)\circ \tau)(u) \,\,\,\,\,\,\,\textup{(definition of a lift)}\\
		&=& (\Phi_{a}\circ \textup{lift}_{\tau}(\psi)\circ F\circ q_{\mathscr{L}})(u)\\
		&=& (\Phi_{a}\circ \textup{lift}_{\tau}(\psi))(p).
	\end{eqnarray*}
	Thus $\varepsilon_{2}(\varepsilon_{1}(a,\psi))=\Phi_{a}\circ \textup{lift}_{\tau}(\psi)$ on $N^{\circ}$, or equivalently, 
	$\varepsilon_{2}(\varepsilon_{1}(a,\psi))=\varepsilon_{3}^{-1}\big(\Phi_{a}\circ \textup{lift}_{\tau}(\psi)\big)$. It follows that 
	$\varepsilon(a,\psi)=\Phi_{a}\circ \textup{lift}_{\tau}(\psi)$. This concludes the proof of the theorem. 
\end{proof}

	In the next two corollaries, $\Phi:\mathbb{T}^{n}\times N\to N$ is a regular torification of a connected dually flat manifold $(M,h,\nabla)$.

\begin{corollary}\label{nfekndkefkndknk}
	For every compatible projection $\kappa:N^{\circ}\to M$, there is a surjective 
	group homomorphism $\varepsilon:\textup{Aut}(N,g)^{\mathbb{T}^{n}}\to \textup{Diff}(M,h,\nabla)$ such that 
	$\kappa\circ \varphi=\varepsilon(\varphi)\circ \kappa$ for all $\varphi\in \textup{Aut}(N,g)^{\mathbb{T}^{n}}$. 
\end{corollary}
\begin{proof}
	Let $\tau$ be a compatible covering map such that $\pi=\kappa\circ \tau$ is a toric factorization. 
	Let $\varphi\in \textup{Aut}(N,g)^{\mathbb{T}^{n}}$ be arbitrary. By Theorem \ref{nfkwndkefnknkwn}, 
	there are $a\in \mathbb{T}^{n}$ and $\psi\in \textup{Diff}(M,h,\nabla)$ such that 
	$\varphi=\Phi_{a}\circ \textup{lift}_{\tau}(\psi)$, and so 
	\begin{eqnarray}\label{neknekfneknk}
		\kappa\circ \varphi=\kappa\circ \Phi_{a}\circ \textup{lift}_{\tau}(\psi)=\kappa\circ \textup{lift}_{\tau}(\psi)=
		\psi\circ \kappa, 
	\end{eqnarray}
	where we have used the fact that compatible projections are $\mathbb{T}^{n}$-invariant and the commutativity of 
	Diagram \eqref{knkdfnknkndknk}. Define the group homomorphism 
	$\varepsilon:\textup{Aut}(N,g)^{\mathbb{T}^{n}}\to \textup{Diff}(M,h,\nabla)$ by $\varepsilon(\Phi_{a}\circ 
	\textup{lift}_{\tau}(\psi))=\psi$ (it is well defined and surjective by Theorem \ref{nfkwndkefnknkwn}).
	In terms of $\varepsilon,$ \eqref{neknekfneknk} can be rewritten as 
	$\kappa\circ \varphi=\varepsilon(\varphi)\circ \kappa$ for all $\varphi\in \textup{Aut}(N,g)^{\mathbb{T}^{n}}$, from 
	which the result follows. 
\end{proof}

\begin{corollary}\label{jekwwnkfnknfknkn}
	Let $\textup{Proj}$ denote the set of compatible projections $\kappa:N^{\circ}\to M$. Then the map 
	\begin{eqnarray}\label{nknkenfknkwnken}
		\textup{Diff}(M,h,\nabla) \times \textup{Proj}\to \textup{Proj}, \,\,\,\,\,\,\,\,(\psi,\kappa)\mapsto \psi\circ \kappa, 
	\end{eqnarray}
	is a free and transitive group action. 
\end{corollary}
\begin{proof}
	The fact that \eqref{nknkenfknkwnken} is a well defined group action of $\textup{Diff}(M,h,\nabla)$ on 
	$\textup{Proj}$ follows easily from Corollary \ref{nfekndkefkndknk}. Freeness follows from the fact that compatible projections 
	are surjective. To see that \eqref{nknkenfknkwnken} 
	is transitive, let $\kappa$ and $\kappa'$ be compatible projections. We know from the general theory that there is $\varphi\in 
	\textup{Aut}(N,g)^{\mathbb{T}^{n}}$ such that $\kappa'=\kappa\circ \varphi$ (see Section \ref{neknkenkfnkn}). 
	By Corollary \ref{nknkenfknkwnken}, there exists $\psi\in \textup{Diff}(M,h,\nabla)$ such that $\kappa\circ \varphi=\psi\circ 
	\kappa$. Thus $\kappa'=\psi\circ \kappa$, which shows transitivity. 
\end{proof}

\section{Weyl group}\label{weeel}

	The group $\textup{Aut}(N,g)$ of holomorphic isometries of a connected and compact K\"{a}hler manifold 
	$(N,g)$, endowed with the compact-open topology, is a compact Lie group whose natural action on $N$ is 
	smooth\footnote{This follows from a classical theorem due to Myers and Steenrod \cite{Myers} and a theorem of 
	Weierstrass in complex geometry. More precisely, the Myers-Steenrod theorem states that 
	the group $\textup{Isom}(M,g)$ of isometries of a connected Riemannian manifold $(M,g)$ is a Lie group with respect to the compact-open topology 
	whose natural action on $M$ is smooth. Moreover, if $M$ is compact, then $\textup{Isom}(M,g)$ is also compact (for a modern 
	proof, see \cite{kobayashi-transformation}, Chapter II). Weierstrass' Theorem states that the space
	$\mathcal{O}(X,Y)$ of holomorphic maps between two complex manifolds $X$ and $Y$ is closed in the space $C(X,Y)$ of 
	continuous maps $X\to Y$ with respect to the compact-open topology. This theorem is 
	usually proved when $X$ is an open subset of $\mathbb{C}^{n}$ and $Y=\mathbb{C}$ (see, e.g., \cite{haslinger}), 
	but it is easy to generalize it to arbitrary complex manifolds. Combining these results, we obtain that the group 
	of holomorphic isometries $\textup{Aut}(N,g)$ of a compact K\"{a}hler manifold $N$ with K\"{a}hler metric $g$ is 
	a closed subgroup of the compact Lie group $\textup{Isom}(N,g)$, and so it is a compact Lie group.}. 
	If $\Phi:\mathbb{T}^{n}\times N\to N$ is an effective holomorphic and isometric torus action, it is not hard to see\footnote{
	This can be proved by arguments similar to those in \cite[Chapter I]{kobayashi-transformation}.} that 
	the image $S$ of $\mathbb{T}^{n}$ under the map $a\mapsto \Phi_{a}$ is a closed $n$-dimensional torus in $\textup{Aut}(N,g)$. 
	It is thus natural to ask whether $S$ is a Cartan subgroup of $\textup{Aut}(N,g)$ (in the 
	sense of Segal), and if so, what is the corresponding Weyl group $W(S)=N(S)/S$. 

%

	The next theorem is the main result of this paper. 

%
%
%
%
%
%
%
%

\begin{theorem}\label{nefknkenknkwndk}
	Let $\Phi:\mathbb{T}^{n}\times N\to N$ be a compact regular torification of a connected dually flat manifold $(M,h,\nabla)$, 
	with K\"{a}hler metric $g$. Let $S\subset \textup{Aut}(N,g)$ be the image of $\mathbb{T}^{n}$ under the map 
	$a\mapsto \Phi_{a}$. Then $S$ is a maximal torus and a Cartan subgroup of $\textup{Aut}(N,g)$. Moreover, given 
	a compatible covering map $\tau:TM\to N^{\circ}$, the map
	\begin{eqnarray}\label{kkfksjfkjkjkjdnjnek}
		\textup{Diff}(M,h,\nabla)\to W(S),\,\,\,\, \psi\mapsto \big[\textup{lift}_{\tau}(\psi)\big]
	\end{eqnarray}
	is a group isomorphism, where $W(S)=N(S)/S$ is the Weyl group (in the sense of Segal) associated to $S$,
	and $[\textup{lift}_{\tau}(\psi)]$ denotes the equivalence class of $\textup{lift}_{\tau}(\psi)$ in $W(S)$. 
\end{theorem}

	Recall that a toric K\"{a}hler manifold $(N,\omega,\Phi,\moment)$ is always a
	torification of its momentum polytope $(\Delta^{\circ},k,\nabla^{k})$, and also a torification of its 
	holomorphic polytope $(\mathbb{R}^{n},h,\nabla^{\textup{flat}})$ (see Section \ref{nknknvknknk}). 
	If in addition the K\"{a}hler metric $g$ is real analytic, then 
	$N$ is regular and Theorem \ref{nefknkenknkwndk} immediately implies the following result. 

\begin{corollary}
	Let $(N,\omega,\Phi,\moment)$ be a toric K\"{a}hler manifold with real analytic K\"{a}hler metric $g$. 
	Let $(\Delta^{\circ},k,\nabla^{k})$ be the corresponding momentum polytope, and let 
	$(\mathbb{R}^{n},h,\nabla^{\textup{flat}})$ be the holomorphic polytope associated to some $p\in N^{\circ}$. 
	If $W$ denotes the Weyl group of $\textup{Aut}(N,g)$ associated to $S=\{\Phi_{a}\,\,\vert\,\,a\in \mathbb{T}^{n}\}$, then
	\begin{eqnarray*}
	W\,\,\cong\,\, \textup{Diff}(\Delta^{\circ},k,\nabla^{k})\,\,\cong \,\,\textup{Diff}(\mathbb{R}^{n},h,\nabla^{\textup{flat}}). 
		\,\,\,\,\,\,\,\,\, \textup{(group isomorphisms)} 
	\end{eqnarray*}
\end{corollary}

	The rest of this section is devoted to the proof of Theorem \ref{nefknkenknkwndk}. 
	Let $M,N,S,\Phi$ and $\tau$ be as in the theorem. 

\begin{lemma}\label{nkenkenfkknk}
	$S$ is a maximal torus in $\textup{Aut}(N,g)$. 
\end{lemma}
\begin{proof}
	Let $T\subseteq \textup{Aut}(N,g)$ be a torus such that 
	$S\subseteq T$. We must show that $S=T$. Because $S$ and $T$ are tori,  
	it suffices to show that $\textup{dim}(S)=\textup{dim}(T)$. 
	The natural action of $T$ on $N$ 
	is effective and symplectic, and since $N$ is simply connected, it is also 
	Hamiltonian. By (P2), this forces $2\,\textup{dim}(T)\leq \textup{dim}(N)$. 
	Since $\textup{dim}(N)=2\,\textup{dim}(\mathbb{T}^{n})=2\,\textup{dim}(S)$, 
	we obtain $\textup{dim}(T)\leq \textup{dim}(S)$. The reverse inequality is obvious. 
\end{proof}

	Recall that $\textup{Aut}(N,g)^{\mathbb{T}^{n}}$ denotes the group of holomorphic isometries of $N$ 
	that are equivariant in the following sense: for each $\varphi\in \textup{Aut}(N,g)^{\mathbb{T}^{n}}$, 
	there is a Lie group isomorphism $\rho:\mathbb{T}^{n}\to \mathbb{T}^{n}$ 
	such that $\varphi\circ \Phi_{a}=\Phi_{\rho(a)}\circ \varphi$ for all $a\in \mathbb{T}^{n}$. 
	Let $N(S)=\{\varphi\in \textup{Aut}(N,g)\,\vert\,\,\varphi\circ S\circ \varphi^{-1}=S\}$ be the normalizer of $S$ in $\textup{Aut}(N,g)$. 

\begin{lemma}\label{nefkkwknefknwknknfenkdfndknk}
	$N(S)=\textup{Aut}(N,g)^{\mathbb{T}^{n}}$.
\end{lemma}
\begin{proof}
	The inclusion $\textup{Aut}(N,g)^{\mathbb{T}^{n}}\subseteq N(S)$ is obvious. For the reverse inclusion, 
	let $\varphi\in N(S)$ be arbitrary. It is easy to see from the definition of $N(S)$ and the fact that 
	$\Phi$ is effective that there is a group isomorphism $\rho:\mathbb{T}^{n}\to \mathbb{T}^{n}$ 
	such that 
	\begin{eqnarray}\label{nenkefknknknnn}
		\varphi\circ \Phi_{a}\circ \varphi^{-1}=\Phi_{\rho(a)} \,\,\,\,\,\,\,\,\Leftrightarrow \,\,\,\,\,\,\,\,
		\varphi\circ \Phi_{a}=\Phi_{\rho(a)}\circ \varphi
	\end{eqnarray}
	for all $a\in \mathbb{T}^{n}$. 
	We claim that $\rho$ is smooth. To see this, let $(\mathscr{L},X,F)$ 
	be a toric parametrization, where $\mathscr{L}$ is the fundamental lattice. 
	In the notation of Section \ref{nekwnkefkwnkk}, let $G:\mathbb{T}^{n}\times M\to N^{\circ}$ be the composition
	\begin{eqnarray*}
		\mathbb{T}^{n}\times M\overset{f}{\longrightarrow} M_{\mathscr{L}}\overset{F}{\longrightarrow} N^{\circ}, 
	\end{eqnarray*}
	where $f([u],p)=q_{\mathscr{L}}(u_{1}X_{1}(p)+...+u_{n}X_{n}(p))=(\Phi_{X})_{[u]}(q_{\mathscr{L}}(0_{p}))$. Then 
	$G$ is a diffeomorphism that is equivariant in the sense that $G\circ \Psi_{a}=\Phi_{a}\circ G$ for all $a\in \mathbb{T}^{n}$, 
	where $\Psi$ is the group action of $\mathbb{T}^{n}$ on $\mathbb{T}^{n}\times M$ given by left mutiplication. 
	Consider the composition $\widetilde{\varphi}:=G^{-1}\circ \varphi\circ G$. It is well-defined because 
	$\varphi(N^{\circ})=N^{\circ}$ (this follows from \eqref{nenkefknknknnn}), and is a smooth diffeomorphism of 
	$\mathbb{T}^{n}\times M$. Furthermore, a direct computation using \eqref{nenkefknknknnn} and the equivariance of $G$ shows that 
	$\widetilde{\varphi}$ is of the form 
	\begin{eqnarray}\label{nceknwdkenkfnkk}
		\widetilde{\varphi}(a,p)=(\rho(a)\phi(p),\psi(p)),  
	\end{eqnarray}
	where $\phi:M\to \mathbb{T}^{n}$ and $\psi:M\to M$ are smooth maps. Let $i:\mathbb{T}^{n}\to \mathbb{T}^{n}\times M$, 
	$a\mapsto (a,p_{0})$, where $p_{0}\in M$ is fixed, and 
	let $\pi:\mathbb{T}^{n}\times M\to \mathbb{T}^{n}$ be the projection onto the first factor. 
	It follows from \eqref{nceknwdkenkfnkk} that $\rho=\pi\circ \Psi_{\phi(p_{0})^{-1}}\circ \widetilde{\varphi}\circ i$, 
	which shows that $\rho$ is smooth. This completes the proof of the claim. It follows from the discussion above and 
	the claim that $\varphi\in \textup{Aut}(N,g)^{\mathbb{T}^{n}}$. 
	Thus $N(S)\subseteq \textup{Aut}(N,g)^{\mathbb{T}^{n}}$.
%
%
%
\end{proof}

	The next step in the proof of Theorem \ref{nefknkenknkwndk} is to show that $\textup{Diff}(M,h,\nabla)$ 
	is finite. We will do this by showing that $\textup{Diff}(M,h,\nabla)$ acts effectively on a finite set. 
	Let $\textup{Aff}(\mathbb{R}^{n})$ be the group of 
	affine transformations of $\mathbb{R}^{n}$. Thus, for each $\varphi\in \textup{Aff}(\mathbb{R}^{n})$, there 
	are a $n\times n$ invertible real matrix $A$ and $B\in \mathbb{R}^{n}$ such that 
	$\varphi(x)=Ax+B$ for all $x\in \mathbb{R}^{n}$. 
	Recall that a subset of $\mathbb{R}^{n}$ is a \textit{convex polytope} if it is the convex hull 
	of finitely many points in $\mathbb{R}^{n}$. 

\begin{lemma}\label{nkdnknefkn}
	Let $\Delta\subset \mathbb{R}^{n}$ be a convex polytope and let $\Gamma$ be a subgroup of $\textup{Aff}(\mathbb{R}^{n})$
	such that $\gamma(\Delta)=\Delta$ for all $\gamma\in \Gamma.$ If the topological interior of $\Delta$ is nonempty, 
	then $\Gamma$ is finite. 
\end{lemma}
\begin{proof}
	Let $\textup{Ext}(\Delta)\subseteq \Delta$ be the set of 
	extreme points of $\Delta$. Thus a point $x\in \Delta$ is in 
	$\textup{Ext}(\Delta)$ if and only if it is not possible to express
	$x$ as a convex combination $(1-t)y+tz$, where $y,z\in \Delta$ and $0<t<1$, except by taking $x=y=z$. 
	We have:
	\begin{enumerate}[(1)]
	\item $\textup{Ext}(\Delta)$ is finite (this follows from \cite{Rockafellar}, Corollary 18.3.1), 
	\item $\Delta$ is the convex hull of $\textup{Ext}(\Delta)$ (Krein-Milman theorem), 
	\item $\gamma(\textup{Ext}(\Delta))=\textup{Ext}(\Delta)$ for all $\gamma\in \Gamma$ (this follows easily 
		from the fact that each $\gamma$ is affine and preserves $\Delta$). 
	\end{enumerate}

	It follows from (3) that $\Gamma$ naturally acts on $\textup{Ext}(\Delta)$, which is a finite set by (1). 
	Thus, to prove the lemma, it suffices to show that the action of $\Gamma$ on $\textup{Ext}(\Delta)$ 
	is effective. So let $\gamma$ be an element of $\Gamma$ such that 
	$\gamma(x)=x$ for all $x\in \textup{Ext}(\Delta)$. For simplicity, we will denote by 
	$\textup{Conv}(S)$ (resp. $\textup{Aff}(S)$) the convex hull (resp. affine hull) of a 
	subset $S\subseteq \mathbb{R}^{n}$. Using (2) and the fact that $\textup{Aff}(\textup{Conv}(S))=\textup{Aff}(S)$ for any 
	nonempty set $S\subseteq \mathbb{R}^{n}$, we see that $\textup{Aff}(\textup{Ext}(\Delta))= \textup{Aff}(\Delta)$. 
	Since the interior of $\Delta$ is nonempty by hypothesis, $\textup{Aff}(\Delta)=\mathbb{R}^{n}$. Therefore 
	$\textup{Aff}(\textup{Ext}(\Delta))=\mathbb{R}^{n}$. This implies that $\textup{Ext}(\Delta)$ contains 
	$n+1$ points $x_{0},...,x_{n}$ that are affinely independent. Because of our hypothesis on $\gamma$, 
	we have $\gamma(x_{i})=x_{i}$ for all $i\in\{0,....,n\}$, and thus $\gamma$ must be the identity map 
	on $\mathbb{R}^{n}$ (see \cite{Rockafellar}, Theorem 1.6). The lemma follows. 
\end{proof}

	By Remark \ref{ndkdnknknknk}(c), $N$ is naturally a K\"{a}hler toric manifold. 
	Let $\moment:N\to \mathbb{R}^{n}$ be an equivariant momentum map and let $(\Delta^{\circ},k,\nabla^{k})$ be 
	the corresponding momentum polytope (see Section \ref{nknknvknknk}).
%
%
	In \cite{molitor-toric}, it is proven that $(\Delta^{\circ},k,\nabla^{k})$ is a dually flat manifold and that there 
	is an isomorphism of dually flat manifolds between $(M,h,\nabla)$ and $(\Delta^{\circ},k,\nabla^{k})$. 
	Therefore the groups $\textup{Diff}(M,h,\nabla)$ and $\textup{Diff}(\Delta^{\circ},k,\nabla^{k})$ are isomorphic. 
	Thus, to prove that $\textup{Diff}(M,h,\nabla)$ if finite, it suffices to show that 
	$\textup{Diff}(\Delta^{\circ},k,\nabla^{k})$ is finite. Before we do so, we make to following 
	observation. 


\begin{lemma}
	$\textup{Diff}(\Delta^{\circ},k,\nabla^{k})=\textup{Diff}(\Delta^{\circ},k,\nabla^{\textup{flat}})$.
\end{lemma}

	This lemma is a special case of the following result, in which 
	we do not presuppose any particular properties on $(M,h,\nabla)$ (like being toric). 

\begin{lemma}
	Let $(M,h,\nabla)$ be a connected dually flat manifold. 
	Then $\textup{Diff}(M,h,\nabla)=\textup{Diff}(M,h,\nabla^{*})$, where 
	$\nabla^{*}$ is the dual connection of $\nabla$ with respect to $h$. 
\end{lemma}
\begin{proof}
	Clearly, it suffices to show one inclusion. So let $\varphi\in \textup{Diff}(M,h,\nabla)$ be arbitrary. 
	We must show that $\varphi$ is affine with respect to $\nabla^{*}$. Since $\varphi$ is a diffeomorphism, this is equivalent 
	to show that $\varphi_{*}(\nabla^{*}_{X}Y)=\nabla^{*}_{\varphi_{*}X}(\varphi_{*}Y)$ for 
	all vector fields $X,Y$ on $M$, where $\varphi_{*}X$ 
	denotes the pushforward of $X$ by $\varphi$, that is, $(\varphi_{*}X)_{p}=\varphi_{*_{\varphi^{-1}(p)}}X_{\varphi^{-1}(p)}$ 
	for all $p\in M$ (see \cite[Chapter VI, Proposition 1.4]{Kobayashi-Nomizu}). To see this, it suffices to show that 
	\begin{eqnarray*}
		h\big(\varphi_{*}(\nabla^{*}_{X}Y),\varphi_{*}Z\big)=h\big(\nabla^{*}_{\varphi_{*}X}(\varphi_{*}Y),\varphi_{*}Z\big)
	\end{eqnarray*}
	for all vector fields $X,Y,Z$ on $M$, which can be done by a direct computation using the formula 
	$A(h(B,C))=h(\nabla_{A}B,C)+h(B,\nabla^{*}_{A}C)$, where $A,B,C$ are vector fields on $M$.  
\end{proof}

\begin{lemma}\label{neknwkenkwnkn}
	$\textup{Diff}(M,h,\nabla)$ is finite.
\end{lemma}
\begin{proof}
	In view of the discussion above, it suffices to prove that $\textup{Diff}(\Delta^{\circ},k,\nabla^{\textup{flat}})$ 
	is finite. Let $\varphi\in \textup{Diff}(\Delta^{\circ},k,\nabla^{\textup{flat}})$ be arbitrary. 
	Because $\varphi$ is affine with respect to the flat connection $\nabla^{\textup{flat}}$, it extends uniquely 
	to an affine map $\widetilde{\varphi}\in \textup{Aff}(\mathbb{R}^{n})$ (see \cite[Chapter VI, Corollary 6.2]{Kobayashi-Nomizu}). 
	We claim that $\widetilde{\varphi}(\Delta)=\Delta$. Indeed, the fact that $\Delta$ is a closed convex set whose 
	topological interior $\Delta^{\circ}$ is nonempty implies that $\Delta$ is the closure of its interior, that is, 
	$\Delta=\textup{cl}(\Delta^{\circ})$ (see \cite{Rockafellar}, Theorem 6.3). It follows that 
	\begin{eqnarray*}
		\widetilde{\varphi}(\Delta)=\widetilde{\varphi}(\textup{cl}(\Delta^{\circ}))=
		\textup{cl}(\widetilde{\varphi}(\Delta^{\circ}))=\textup{cl}(\varphi(\Delta^{\circ}))= 
		\textup{cl}(\Delta^{\circ})=\Delta,
	\end{eqnarray*}
	where we have used the fact that $\widetilde{\varphi}$ is a homeomorphism. This proves the claim. 
	The latter implies that $\textup{Diff}(\Delta^{\circ},k,\nabla^{\textup{flat}})$ is isomorphic 
	to a subgroup of $\textup{Aff}(\mathbb{R}^{n})$ that leaves invariant $\Delta$. By Lemma \ref{nkdnknefkn}, this 
	group is finite.  
\end{proof}

\begin{proof}[Proof of Theorem \ref{nefknkenknkwndk}]
	Let $S=\{\Phi_{a}\,\,\vert\,\,a\in \mathbb{T}^{n}\}\subset \textup{Aut}(N,g).$
	The fact that $\textup{Diff}(M,h,\nabla)\to N(S)/S$, $\psi\mapsto [\textup{lift}_{\tau}(\psi)]$ is a group isomorphism 
	follows from Theorem \ref{nfkwndkefnknkwn} and Lemma \ref{nefkkwknefknwknknfenkdfndknk}. 
	By Lemma \ref{nkenkenfkknk}, $S$ is a torus and hence it contains an element whose powers are dense in $S$. 
	The index of $S$ in its normalizer $N(S)$ is finite because $N(S)/S$ is isomorphic to $\textup{Diff}(M,h,\nabla)$, 
	which is finite by Lemma \ref{neknwkenkwnkn}. 
	Therefore $S$ is a Cartan subgroup of $\textup{Aut}(N,g)$. 
\end{proof}

\section{The case of exponential families}\label{nknkeknkenfkn}

	In this section, we consider the special case in which the dually flat manifold $(M,h,\nabla)$ is an exponential family 
	$\mathcal{E}$ defined over a finite set $\Omega$. Our objective is to show Proposition \ref{nfekwnknfeknk} below, 
	which gives an algebraic characterization of the group of affine isometries of $\mathcal{E}$. 

	Let $\Omega=\{x_{1},...,x_{m}\}$ be a finite set endowed with the counting measure, and let 
	$\mathcal{E}$ be an exponential family defined over $\Omega$, with elements of the form 
        \begin{eqnarray}\label{nekwnkenkfnk}
		p(x;\theta)=p_{\theta}(x)=e^{C(x)+\langle \theta,F(x)\rangle-\psi(\theta)}, 
	\end{eqnarray}
	where $\theta\in \mathbb{R}^{n}$, 
	$C:\Omega\to \mathbb{R}$, $F=(F^{1},...,F^{n}):\Omega\to \mathbb{R}^{n}$, $\langle u,v\rangle=u_{1}v_{1}+...+u_{n}v_{n}$ 
	is the usual Euclidean pairing in $\mathbb{R}^{n}$ and $\psi:\mathbb{R}^{n}\to \mathbb{R}$. 
	It is assumed that the functions $1,F^{1},...,F^{n}$ are linearly independent. 
	Let $h_{F}$ and $\nabla^{(e)}$ be the Fisher metric and exponential connection on $\mathcal{E}$, 
	respectively.

	For simplicity, we write $C_{i}=C(x_{i})$ and 
	$F_{i}=F(x_{i})=(F_{i}^{1},...,F_{i}^{n})$. Note that the condition $\sum_{x\in \Omega}p(x;\theta)=1$ implies 
	\begin{eqnarray}\label{nknknrkfenknk}
		\psi(\theta)=\ln\bigg(\sum_{i=1}^{m}e^{C_{i}+\langle \theta,F_{i}\rangle }\bigg)
	\end{eqnarray}
	for all $\theta\in \mathbb{R}^{n}$. It is well-known that the coordinate expression for the Fisher metric 
	in the $\theta$-coordinates is the Hessian of $\psi$, that is (see \cite{Amari-Nagaoka}),
	\begin{eqnarray}\label{ndknknknsknkn}
		(h_{F})_{ij}(\theta)=h_{F}\bigg(\dfrac{\partial}{\partial \theta_{i}},
	\dfrac{\partial}{\partial \theta_{j}}\bigg)=\dfrac{\partial^{2}\psi}{\partial \theta_{i}\partial \theta_{j}}(\theta).
	\end{eqnarray}
	Let $\mathbb{S}_{m}$ denote the permutation group of $\{1,...,m\}$. 

\begin{definition}
	We shall say that a diffeomorphism $\varphi$ of $\mathcal{E}$ is a \textit{permutation} 
	if there is $\sigma\in \mathbb{S}_{m}$ such that 
	$\varphi(p)(x_{i})=p(x_{\sigma(i)})$ for all $p\in \mathcal{E}$ and all $i\in \{1,...,m\}$. In this case, we write 
	$\varphi=\varphi_{\sigma}$. 
\end{definition}

	We shall denote by $\textup{Perm}(\mathcal{E})$ the group of permutations of $\mathcal{E}$. 
	When $F:\Omega\to \mathbb{R}^{n}$ is injective, a simple verification shows that $\varphi_{\sigma}=\varphi_{\sigma'}$ 
	implies $\sigma=\sigma'$. Thus, in this case, $\textup{Perm}(\mathcal{E})$ can be regarded as a subgroup of 
	$\mathbb{S}_{m}$ via the group homomorphism $\varphi_{\sigma}\mapsto \sigma^{-1}$. 

\begin{lemma}\label{nfekwcdnmnkdnkdnkrnvknk}
	$\textup{Perm}(\mathcal{E})\subseteq \textup{Diff}(\mathcal{E},h_{F},\nabla^{(e)})$. 
\end{lemma}
\begin{proof}
	Let $\sigma\in \mathbb{S}_{m}$ and $\varphi\in \textup{Diff}(\mathcal{E})\cong\textup{Diff}(\mathbb{R}^{n})$. Suppose that
	$p_{\varphi(\theta)}(x_{i})=p_{\theta}(x_{\sigma(i)})$ for all $\theta\in \mathbb{R}^{n}$ 
	and all $i\in \{1,...,m\}$. Write $\varphi(\theta)=(\varphi^{1}(\theta),...,\varphi^{n}(\theta))$. 
	In view of \eqref{nekwnkenkfnk}, we have 
	\begin{eqnarray*}
		C_{i}+\sum_{j=1}^{n}\varphi^{j}(\theta) F^{j}_{i}-\psi(\varphi(\theta))
		=C_{\sigma(i)}+\sum_{j=1}^{n}\theta_{j} F^{j}_{\sigma(i)}-\psi(\theta)
	\end{eqnarray*}
	for all $i\in \{1,...,m\}$ and all $\theta\in \mathbb{R}^{n}$. 
	Taking the derivative with respect to $\theta_{a}$, and then with respect to $\theta_{b}$, we get 
        \begin{eqnarray}
		&& \sum_{j=1}^{n}\dfrac{\partial \varphi^{j}}{\partial \theta_{a}}(\theta) F^{j}_{i}
		-\sum_{k=1}^{n}\dfrac{\partial \psi}{\partial \theta_{k}}(\varphi(\theta))
			\dfrac{\partial \varphi^{k}}{\partial \theta_{a}}(\theta)
			=F_{\sigma(i)}^{a}-\dfrac{\partial \psi}{\partial \theta_{a}}(\theta) \nonumber\\
		&\Rightarrow  &  \sum_{j=1}^{n}\dfrac{\partial^{2}\varphi^{j}}{\partial \theta_{a}\partial \theta_{b}}F_{i}^{j}
			-\sum_{k,l=1}^{n}\dfrac{\partial^{2}\psi}{\partial \theta_{k}\partial 
			\theta_{l}}(\varphi(\theta))\dfrac{\partial\varphi^{l}}{\partial \theta_{b}}
			\dfrac{\partial \varphi^{k}}{\partial \theta_{a}}-\sum_{k=1}^{n}\dfrac{\partial \psi}{\partial \theta_{k}}
			(\varphi(\theta))
			\dfrac{\partial^{2}\varphi^{k}}{\partial \theta_{a}\partial \theta_{b}} 
			= -\dfrac{\partial^{2}\psi}{\partial \theta_{a}\partial \theta_{b}}\nonumber\\
		&\Rightarrow& 
			\sum_{j=1}^{n}\dfrac{\partial^{2}\varphi^{j}}{\partial \theta_{a}\partial \theta_{b}}F_{i}^{j}
			=\underbrace{\sum_{k,l=1}^{n}(h_{F})_{kl}(\varphi(\theta))\dfrac{\partial\varphi^{l}}{\partial \theta_{b}}
			\dfrac{\partial \varphi^{k}}{\partial \theta_{a}}+\sum_{k=1}^{n}\dfrac{\partial \psi}{\partial 
			\theta_{k}}(\varphi(\theta))
			\dfrac{\partial^{2}\varphi^{k}}{\partial \theta_{a}\partial \theta_{b}}  -(h_{F})_{ab}(\theta)}_{=:T_{ab}(\theta)}.
			\nonumber
			\\\label{nfenkrnkenk}
	\end{eqnarray}
	where we have used \eqref{ndknknknsknkn}. It follows that for all $a,b,i,i'\in \{1,...,m\}$ and all $\theta\in \mathbb{R}$, 
        \begin{eqnarray*}
		\sum_{j=1}^{n}\dfrac{\partial^{2}\varphi^{j}}{\partial \theta_{a}\partial \theta_{b}}(F_{i}^{j}-F_{i'}^{j})=
		T_{ab}(\theta)-T_{ab}(\theta)=0. 
	\end{eqnarray*}
	Because the functions $1,F^{1},...,F^{n}:\Omega\to \mathbb{R}$ are linearly independent, this implies 
        \begin{eqnarray*}
		\dfrac{\partial^{2}\varphi^{j}}{\partial \theta_{a}\partial \theta_{b}}(\theta)=0
	\end{eqnarray*}
	for all $a,b,j\in \{1,...,n\}$ and all $\theta\in \mathbb{R}^{n}$. Thus $\varphi$ is of the form $\varphi(\theta)=A\theta+B$, 
	where $A$ is a $n\times n$ real matrix and $B\in \mathbb{R}^{n}$. This shows that $\varphi$ is affine with respect to $\nabla^{(e)}$. 
	Using this, we can rewrite \eqref{nfenkrnkenk} as 
        \begin{eqnarray*}
		\sum_{k,l=1}^{n}(h_{F})_{kl}(\varphi(\theta))\dfrac{\partial\varphi^{l}}{\partial \theta_{b}}
			\dfrac{\partial \varphi^{k}}{\partial \theta_{a}}=(h_{F})_{ab}(\theta),
	\end{eqnarray*}
	which shows that $\varphi$ is an isometry. The lemma follows.
\end{proof}

\begin{lemma}\label{nenkdfeknfknfk}
	If $F:\Omega\to \mathbb{R}^{n}$ is injective, 
	then $\textup{Diff}(\mathcal{E},h_{F},\nabla^{(e)})\subseteq \textup{Perm}(\mathcal{E})$.  
\end{lemma}
\begin{proof}
	Let $\varphi\in \textup{Diff}(\mathcal{E},h_{F},\nabla^{(e)})$ be arbitrary. 
	Since $(h_{F})_{ab}$ is the Hessian of $\psi$, we have
	\begin{eqnarray*}
		\lefteqn{\dfrac{\partial^{2}}{\partial \theta_{a}\partial \theta_{b}}\big[\psi(\theta)-\psi(\varphi(\theta))\big]}\\
		&=& (h_{F})_{ab}- 
		\underbrace{\sum_{k,l=1}^{n}(h_{F})_{kl}(\varphi(\theta))\dfrac{\partial\varphi^{l}}{\partial \theta_{b}}
			\dfrac{\partial \varphi^{k}}{\partial \theta_{a}}}_{=(h_{F})_{ab}}
			-\sum_{l=1}^{n}\dfrac{\partial \psi}{\partial \theta_{l}}(\varphi(\theta))
			\underbrace{\dfrac{\partial^{2}\varphi^{l}}{\partial \theta_{a}\partial \theta_{b}}}_{=0}\\
		&=& (h_{F})_{ab}-(h_{F})_{ab}=0.
		\end{eqnarray*}
	Note that we have used the facts that $\varphi$ is isometric and affine. 
	It follows that there is $u\in \mathbb{R}^{n}$ and $c\in \mathbb{R}$ such that
	\begin{eqnarray}\label{nkwdnkfeknkn}
		\psi(\theta)-\psi(\varphi(\theta))=\langle u,\theta\rangle+c
	\end{eqnarray}
	for all $\theta\in \mathbb{R}$. In view of \eqref{nknknrkfenknk}, this can be rewritten as 
	\begin{eqnarray*}
		\ln\bigg[\sum_{i=1}^{m}e^{C_{i}+\langle \theta,F_{i}\rangle}\bigg]
		-\ln\bigg[\sum_{i=1}^{m}e^{C_{i}+\langle \varphi(\theta),F_{i}\rangle}\bigg]=\langle u,\theta\rangle+c
	\end{eqnarray*}
	from which we obtain 
	\begin{eqnarray}\label{ncekwdnkenfknk}
		\sum_{i=1}^{m}e^{C_{i}+\langle \theta,F_{i}\rangle}=
		\sum_{i=1}^{m}e^{C_{i}+\langle \varphi(\theta),F_{i}\rangle+\langle u,\theta\rangle+c}.
	\end{eqnarray}
	Since $\varphi$ is affine with respect to $\nabla^{(e)}$, there 
	are an invertible $n\times n$ real matrix $A$ and $B\in \mathbb{R}^{n}$ such that $\varphi(\theta)=A\theta+B$ 
	for all $\theta\in \mathbb{R}$. Thus \eqref{ncekwdnkenfknk} can be rewritten as
	\begin{eqnarray}\label{nceknkenfknk}
		\sum_{i=1}^{m}e^{C_{i}+\langle \theta,F_{i}\rangle}= 
		\sum_{i=1}^{m}e^{C_{i}+\langle B,F_{i}\rangle +c +\langle \theta,A^{T}F_{i}+u\rangle }, 
	\end{eqnarray}
	where $A^{T}$ is the transpose of $A$. Now, the fact that $F$ is injective and $A$ is invertible implies that 
	$F_{i}\neq F_{j}$ and $A^{T}F_{i}+u\neq A^{T}F_{j}+u$ whenever $i\neq j$. 
	It follows from this, \eqref{nceknkenfknk} and Lemma \ref{dnnknefknkefnk} (see below) that 
	there is $\sigma\in \mathbb{S}_{m}$ such that
	\begin{eqnarray}\label{nfeknknefknkn}
	\left\lbrace
	\begin{array}{lll}
		F_{\sigma(i)} &=&  A^{T}F_{i}+u, \\
		C_{\sigma(i)} &=&  C_{i}+\langle B,F_{i}\rangle +c 
	\end{array}
	\right.  
	\end{eqnarray}
	for all $i\in \{1,...,m\}$. These equations, together with \eqref{nekwnkenkfnk} and \eqref{nkwdnkfeknkn}, imply that 
	\begin{eqnarray*}
		p_{A\theta+B}(x_{i})&=& e^{C_{i}+\langle A\theta+B,F_{i}\rangle-\psi(A\theta+B)}\\
		&=& e^{C_{i}+\langle A\theta+B,F_{i}\rangle+\langle u,\theta\rangle+c-\psi(\theta)}\\ 
		&=& e^{{C_{i}+\langle B,F_{i}\rangle +c}+\langle \theta,
		{A^{T}F_{i}+u}\rangle-\psi(\theta)}\\ 
		&=& e^{C_{\sigma(i)}+\langle \theta,F_{\sigma(i)}\rangle-\psi(\theta)}\\
		&=& p_{\theta}(x_{\sigma(i)}). 
	\end{eqnarray*}
	Thus $\varphi$ is a permutation of $\mathcal{E}$. The lemma follows.  
\end{proof}

\begin{lemma}\label{dnnknefknkefnk}
	If $v_{1},...,v_{m}$ are distinct vectors in $\mathbb{R}^{n}$, then the functions $e^{\langle x,v_{1}\rangle}$, 
	..., $e^{\langle x,v_{m}\rangle}$ ($x\in \mathbb{R}^{n}$) are linearly independent over $\mathbb{R}$.
\end{lemma}
\begin{proof}
	Let $\lambda_{1},...,\lambda_{m}\in \mathbb{R}$ be such that 
	\begin{eqnarray}\label{neknekekfnk}
		\lambda_{1}e^{\langle x,v_{1}\rangle}+...+ \lambda_{m}e^{\langle x,v_{m}\rangle}=0
	\end{eqnarray}
	for all $x\in \mathbb{R}^{m}$. We must show that $\lambda_{1}=...=\lambda_{m}=0$. 
	Given a pair of distinct indices $i,j\in \{1,...,m\}$, let 
	$E_{ij}=\{u\in \mathbb{R}^{n}\,\,\vert\,\,\langle u,v_{i}-v_{j}\rangle=0\}$.
	Since $v_{i}-v_{j}\neq 0$, $\textup{dim}(E_{ij})=n-1$, and so $E_{ij}$ has measure zero in $\mathbb{R}^{n}$. 
	Since a countable union of sets of measure zero has measure zero, $E=\cup_{i\neq j}E_{ij}$ 
	has measure zero in $\mathbb{R}^{n}$. This implies $\mathbb{R}^{n}-E\neq \emptyset.$  
	Let $\xi\in \mathbb{R}^{n}-E$	be arbitrary. This choice guarantees that $\langle \xi,v_{i}\rangle
	\neq \langle \xi,v_{j}\rangle$ whenever $i\neq j$. Substituting $x=t\xi$, where $t\in \mathbb{R}$, 
	into \eqref{neknekekfnk} yields 
	\begin{eqnarray}\label{neknekekdndnmnfmfnk}
		\lambda_{1}e^{t\langle \xi,v_{1}\rangle}+...+ \lambda_{m}e^{t\langle \xi,v_{m}\rangle}=0.
	\end{eqnarray}
	Now, it is well known that the lemma is true when $n=1$, and thus the 
	functions $e^{t\langle \xi,v_{1}\rangle}$, ..., 
	$e^{t\langle \xi,v_{m}\rangle}$ are linearly independent over $\mathbb{R}$. In view of \eqref{neknekekdndnmnfmfnk}, 
	this implies that $\lambda_{1}=...=\lambda_{m}=0$, as desired.  
\end{proof}

	The following result is an immediate consequence of Lemmas $\ref{nfekwcdnmnkdnkdnkrnvknk}$ 
	and \ref{nenkdfeknfknfk}.

\begin{proposition}\label{nfekwnknfeknk}
	Let $\mathcal{E}$ be an exponential family defined over a finite set $\Omega=\{x_{1},...,x_{m}\}$, 
	with elements of the form $p(x;\theta)=e^{C(x)+\langle \theta,F(x)\rangle-\psi(\theta)}$, 
	where $\theta\in \mathbb{R}^{n}$, $C:\Omega\to \mathbb{R}$, $F:\Omega\to \mathbb{R}^{n}$, 
	and $\psi:\mathbb{R}^{n}\to \mathbb{R}$. If $F:\Omega\to \mathbb{R}^{n}$ is injective, 
	then $\textup{Diff}(\mathcal{E},h_{F},\nabla^{(e)})=\textup{Perm}(\mathcal{E})$.  
\end{proposition}

\begin{example}
	If $\mathcal{E}=\mathcal{P}_{n}^{\times}$ (see Example \ref{exa:5.5}), then $\textup{Perm}(\mathcal{P}_{n}^{\times})\cong 
	\mathbb{S}_{n}$, and so $\textup{Diff}(\mathcal{P}_{n}^{\times},h_{F},\nabla^{(e)})\cong \mathbb{S}_{n}$ 
	by Proposition \ref{nfekwnknfeknk}.
\end{example}

\begin{example}\label{nfeknkneknknk}
	Suppose $\textup{dim}(\mathcal{E})=1$ and $F(x_{i})<F(x_{j})$ whenever $i<j$. 
	Then $\textup{Diff}(\mathcal{E},h_{F},\nabla^{(e)})$ is either trivial, or it is $\{\textup{Id}_{\mathcal{E}},\varphi\}$, 
	where $\textup{Id}_{\mathcal{E}}$ is the identity map on $\mathcal{E}$ and $\varphi(p)(x_{i})=p(x_{m-i+1})$. 
	To see this, let $\varphi_{\sigma}\in \textup{Diff}(M,h_{F},\nabla^{(e)})$ 
	be arbitrary, where $\sigma\in \mathbb{S}_{m}$. 
	By inspection of the proof of Lemma \ref{nenkdfeknfknfk} (more specifically \eqref{nfeknknefknkn}), the permutation 
	$\sigma$, regarded as a map $\{1,...,m\}\to \{1,...,m\}$, is strictly increasing or strictly decreasing, 
	and thus $\sigma(k)=k$ or $\sigma(k)=m-k+1$ for all $k\in \{1,...,m\}$.  	
\end{example}

\begin{example}
	When $\mathcal{E}=\mathcal{B}(n)$ (see Example \ref{exa:5.6}), it is easy to check that $\varphi(p)(k)=p(n-k)$ is a permutation of 
	$\mathcal{B}(n)$. Thus $\textup{Perm}(\mathcal{B}(n))$ is not trivial. By Example \ref{nfeknkneknknk},
	$\textup{Diff}(\mathcal{B}(n),h_{F},\nabla^{(e)})\cong \mathbb{S}_{2}$.
\end{example}


	We end this section with a simple observation which is a direct consequence of the 
	proof of Lemma \ref{nenkdfeknfknfk} (more specifically \eqref{nfeknknefknkn}) 
	and Proposition \ref{nfekwnknfeknk}. Let $\mathcal{A}_{\mathcal{E}}$ be the vector space 
	of functions $X:\Omega\to \mathbb{R}$ generated 
	by $1,F^{1},...,F^{n}$. 
\begin{lemma}\label{nceknfkrnkenfekn}
	Suppose $F$ injective. Then the group $\textup{Diff}(\mathcal{E},h_{F},\nabla^{(e)})=\textup{Perm}(\mathcal{E})$ 
	acts linearly on $\mathcal{A}_{\mathcal{E}}$ via the formula
	$\big[\varphi_{\sigma}\cdot X\big](x_{i})= X(x_{\sigma(i)})$.
\end{lemma}
	
	Physical applications of this result are discussed in \cite{molitor-spectral}.

\begin{footnotesize}\bibliography{bibtex}\end{footnotesize}
\end{document}